\documentclass[11pt]{amsart}
\usepackage{graphicx}
\usepackage{amssymb}
\usepackage{epstopdf}
\usepackage{microtype}

\newtheorem{theorem}{Theorem}[section]
\newtheorem{lemma}[theorem]{Lemma}

\newtheorem{corollary}[theorem]{Corollary}

\newtheorem{remark}[theorem]{Remark}
\newtheorem{definition}{Definition}[section]

\def\R{{\mathbb R}}

\def\cL{{\mathcal L}}
\def\cM{{\mathcal M}}

\def\a{\alpha}
\def\b{\beta}
\def\e{\varepsilon}
\def\d{\delta}

\def\G{\Gamma}
\def\k{\kappa}
\def\l{\lambda}
\def\L{\Lambda}
\def\m{\mu}
\def\n{\nabla}
\def\p{\partial}
\def\r{\rho}
\def\s{\sigma}
\def\t{\tau}

\def\W{\Omega}

\def\1{\left(}
\def\2{\right)}
\def\3{\left\{}
\def\4{\right\}}
\def\8{\infty}

\def\sm{\setminus}
\def\ss{\subseteq}

\title{Regularity for solutions of non local, non symmetric equations}
\author[H. Chang Lara]{H\'ector Chang Lara}
\address{%
University of Texas at Austin\\
Department of Mathematics\\
1 University Station C1200\\
Austin, TX 78712
}
\email{hchang@math.utexas.edu}

\author[G. D\'avila]{Gonzalo D\'avila}
\address{%
University of Texas at Austin\\
Department of Mathematics\\
1 University Station C1200\\
Austin, TX 78712
}
\email{gdavila@math.utexas.edu}

\begin{document}

\begin{abstract}
We study the regularity for solutions of fully nonlinear integro differential equations with respect to nonsymmetric kernels. More precisely, we assume that our operator is elliptic with respect to a family of integro differential linear operators where the symmetric part of the kernels have a fixed homogeneity $\sigma$ and the skew symmetric part have strictly smaller homogeneity $\tau$. We prove a weak ABP estimate and $C^{1,\alpha}$ regularity. Our estimates remain uniform as we take $\sigma \to 2$ and $\tau \to 1$ so that this extends the regularity theory for elliptic differential equations with dependence on the gradient.
\end{abstract}
%\subjclass{Primary 35J60; Secondary 47G20}
%\keywords{Integro-differential equations; Viscosity solutions; A priori estimates; Generalized drift}
\maketitle

%%%%%%%%%%%%%%%%%%%%%%%%%%
%%%%%% INTRODUCTION %%%%%%
%%%%%%%%%%%%%%%%%%%%%%%%%%

\section{Introduction}

We are interested in studying integro differential equations that arise when studying discontinuous stochastic processes. By the L\`evy-Khintchine formula, the generator of an $n$-dimensional L\`evy process is given by
\begin{align*} 
Lu(x)=&\sum_{i,j}a_{i,j}u_{i,j}+\sum_ib_iu_i\\
& +\int\limits_{\R^n}(u(x+y)-u(x)-\n u(x)\cdot y\chi_{B_1}(y))d\mu_x(y),
\end{align*}
where $\m$ is a positive measure such that $\int|y|^2/(|y|^2+1)d\mu(y)<\infty$. The first and second term corresponds to the diffusion and drift part, and the third one correspond to the jump. The effect of first term is already well understood as it regularizes the solution. The type of equations that we will study come from processes with only the jump part,
\begin{align}\label{LK}
Lu(x)=\int\limits_{\R^n}(u(x+y)-u(x)-\n u(x)\cdot y\chi_{B_1}(y))d\mu_x(y).
\end{align}
More general than the linear operator are the fully non linear ones, which are also important in stochastic control as seen in \cite{Soner}. For example, a convex type of equation takes the form,
\begin{align}\label{ecuacionsup}
Iu(x)=\sup\limits_{\a}L_\a u(x).
\end{align}
Equation \eqref{ecuacionsup} can be seen as a one player game, for which he can choose different strategies at each step to maximize the expected value of some function at the first exit point of the domain. A natural extension for \eqref{ecuacionsup}, when there are two players competing is 
\begin{align*}
Iu(x)=\inf_\b\sup_\a L_{\a\b} u (x).
\end{align*}

We are mainly interested in studying interior regularity for solutions of 
\begin{align}\label{ec1}
Iu(x)=f(x),  \ \ \hbox{in} \ \Omega,
\end{align}
for $f$ continuous, $\Omega$ a given domain and $I$ a fully non linear operator of fractional order to be defined in the next section. In \cite{S1} the regularity for this type of problem was already established by using analytic techniques. However those estimates blow up as the order of the equation goes to the classical one, so it was expected that better estimates could be possible. Those results are more elaborated and presented in \cite{C1}, \cite{C2} and \cite{C3} in the case that the kernels are symmetric. We remove this symmetry hypothesis of the kernel and are able to obtain $C^\a$ regularity and $C^{1,\a}$ regularity for translation invariant equations.

The paper is divided as follows. In Section \ref{VSP} we give most of the relevant definitions and point out some important examples to keep in mind. Specifically we will introduce the notions of fully non linear, non local operartors, ellipticity and viscosity solution. In Section \ref{SR} we state the main results of this work, which is $C^{\a}$ and $C^{1,\a}$ regularity for solutions of equations of the form \eqref{ec1} under different hypothesis on the kernels. In section \ref{CPE} we study the basic stability properties of the elliptic integro differential operators, the comparison principle and prove existence of solution of the Dirichlet problem by using Perron's method. Sections \ref{ABPp} and \ref{PE} are the core of this paper. In section \ref{ABPp} we prove a weak ABP estimate which combined with a rescaling argument will allow us to prove, in Section \ref{PE}, a point estimate lemma. On section \ref{Ho} we deal with the H\"older regularity by applying the previous point estimate to show a geometric decay of the oscillation of the solution. Finally in section \ref{CoA} we use that, for translation invariant equations, the incremental quotients are also solutions of equations in the same ellipticity class in order to show H\"older regularity for the first derivatives.
 
%%%%%%%%%%%%%%%%%%%%%%%%%%%%%%%%%%%%%%%%%%%%%%%%%%%%%%%
%%%%%% Viscosity Solutions and Maximal operators %%%%%%
%%%%%%%%%%%%%%%%%%%%%%%%%%%%%%%%%%%%%%%%%%%%%%%%%%%%%%%

\section{Preliminaries and Viscosity Solutions}\label{VSP}

In this work we restrict ourselves to measures $d\m_x = K(x,y)dy$. From equation \eqref{LK} we formally can write
\begin{align}\label{dc}
Lu(x)=\int\d_e(u,x;y)K_e(x;y)dy+\int\d_o(u,x;y)K_o(x;y)dy+b(x)\cdot \n u(x),
\end{align}
where 
\begin{align*}
\d_e(u,x;y)&=u(x+y)+u(x-y)-2u(x),\\
\d_o(u,x;y)&=u(x+y)-u(x-y).
\end{align*}
$K_{e,o}$ are the even and odd part of $K$ with respect to $y$ and $b$ is a vector valued function given by
\[
 b(x)=\int\limits_{B_1}K_o(x;y)ydy.
\]
Notice that if the total kernel $K$ is even the last two terms in \eqref{dc} disappear. This was convenient in \cite{C1} as these bring additional difficulties with the scaling as can be noticed in \cite{KAL1}. 

The second term can be considered as a drift term, in the sense that has a ``direction'', by $K_o$ being odd. If the singularity of $K_o$ at the origin is of order $n+\t$, with $\t\to 1^-$, then this integral becomes a gradient term. For this reason, one can consider studying the regularizing effect of the first two terms. The linear operators we are interested are always of the form,
\begin{align}
\label{eqgeneral} Lu(x) &= P.V.\int\limits_{\R^n}(u(x+y)-u(x))K(x;y)dy,\\
 &:= \lim_{\e\to0} \int_{\R^n\sm B_\e}(u(x+y)-u(x))K(x;y)dy. \nonumber
\end{align}

%%% Integrability condition %%%

\subsection{Integrability conditions.}

We want here to make sense of the decomposition \eqref{dc}. All we need for that is that the kernels are not too singular whenever $u \in C^{1,1}(x_0)$. The following definition is the same as in \cite{C1}.

\begin{definition}
We say that a function $u$ is $C^{1,1}$ at the point $x_0$ and write $u \in C^{1,1}(x_0)$ if and only if there exists a vector $v \in \R^n$ and a number $M>0$ such that
\[
|u(x+y)-u(x)-v\cdot y| < M|y|^2 \text{ for $|y|$ small enough.}
\]
\end{definition}

This implies in particular that $|\d_e(u,x_0;y)| = O(|y|^2)$ and $|\d_o(u,x_0;y)| = O(|y|)$ as $|y|$ goes to zero.

With this notion at hand we ask for the kernel $K$, when decomposed in its symmetric and skew symmetric parts, $K=K_e+K_o$ respectively, to satisfy the following integrability conditions,
\begin{align}
\label{intconde}\int\frac{|y|^2}{|y|^2+1}|K_e(y)|dy<\infty,\\
\label{intcondo}\int\frac{|y|}{|y|+1}|K_o(y)|dy<\infty.
\end{align}
These conditions allow us to write rigorously
\[
Lu(x) = \int\d_e(u,x;y)K_e(y)dy+\int\d_o(u,x;y)K_o(y)dy,
\]
for $u \in C^{1,1}(x)\cap L^\8(\R^n)$.

We say that a family $\cL$ of linear operators satisfy the integrability conditions uniformly when the upper bounds in \eqref{intconde} and \eqref{intcondo} can be taken independent of $L \in \cL$.

%%% Non linear operators %%%

\subsection{Non linear, non local operators.}

Before defining what will be for us a fully non linear non local operator we present some examples to keep in mind. They are constructed from the linear operators in \eqref{eqgeneral}. 
\begin{align}
\label{eqinfsup}	\text{(Inf-sup type) } Iu(x) &= \inf_\b\sup_\a L_{\a,\b} u(x),\\
\label{moperator1} 	\text{(Maximal) } \cM^{+}_{\cL}u(x)&=\sup\limits_{L\in\cL}Lu(x),\\
\label{moperator2} 	\text{(Minimal) } \cM^{-}_{\cL}u(x)&=\inf\limits_{L\in\cL}Lu(x).
\end{align}

\begin{definition}\label{nonlocalop}
We say that $I$ is a non local fully non linear operator if it satisfies the following:
\begin{enumerate}
\item[(i)] If $u$ is any bounded $C^{1,1}(x)$ function then $Iu(x)$ is well defined.
\item[(ii)] If $u\in C^{1,1}(\W)$ for some open set $\W \ss \R^n$, then $Iu$ is a continuous function in $\W$.
\end{enumerate}
\end{definition}

Our examples satisfy immediately (i) in the definition above. In order to have the continuity stated in (ii) we need to check a uniform integrability condition in the kernels.

\begin{lemma}
Let $I$ be of the form \eqref{eqinfsup} where $\cL = \{L_{\a,\b}\}$ satisfy the integrability conditions \eqref{intconde} and \eqref{intcondo} uniformly. Then $Iu \in C(\W)$ for every $u \in C^{1,1}(\W)$.
\end{lemma}

\begin{proof}
We need to prove that $L_{\a,\b}u$ are equicontinuous over compact sets of $\W$ in order to conclude by Arzela Ascoli's Theorem. Fix $\d>0$ and lets work over the points $x\in\W$ that are at least $\d$ away from $\R^n\sm\W$.

Let $L_{\a,\b}$ has associated the kernels $K_{\a,\b}(y) = (K_{\a,\b})_e(y) + (K_{\a,\b})_o(y)$, decomposed in its symmetric and skew symmetric parts. Because $u \in C^{1,1}(\W)$ we can write for $x \in \W$, 
\begin{align*}
L_{\a,\b}u(x) &= \int \d_e(u,x;y)(K_{\a,\b})_e(y)dy + \int \d_o(u,x;y)(K_{\a,\b})_o(y)dy,\\
&= \int_{B_r} \d_e(u,x;y) (K_{\a,\b})_e(y)dy + \int_{\R^n\sm B_r} \d_e(u,x;y) (K_{\a,\b})_e(y)dy,\\
&{} + \int_{B_r} \d_o(u,x;y) (K_{\a,\b})_o(y)dy + \int_{\R^n\sm B_r} \d_o(u,x;y) (K_{\a,\b})_o(y)dy.
\end{align*}

The first and third integrals can be smaller than any $\e>0$ if $r$ is small enough. Use that $u$ is $C^{1,1}$ to get that $|\d_e(u,x;y)| \leq C|y|^2$ and $|\d_o(u,x;y)| \leq C|y|$ if $r<\d$ and for some constant $C$ independent of $x$. By the integrability condition and the absolute continuity of the integral we get that, for even smaller radius $r$, the aforementioned terms are smaller than $\e$, independently of $x$ and $L_{\a,\b}$.

Now if we fix a radius $r$, we get that the second and fourth terms are equicontinuous in $x$. For this we just need to apply Lemma 4.1 in \cite{C1}.

As a consequence of the previous two paragraphs, we obtain that the difference $|L_{\a,\b}u(x)-L_{\a,\b}u(x')|$ is arbitrarily small when $|x-x'|$ is sufficiently small, independently of $x$,$x'$ (both at least $\d$ away from $\R^n \sm \W$) and $L_{\a,\b}$.
\end{proof}

%%% Extremal op %%%

\subsection{Extremal operators comparable to the fractional laplacians.}

An important family, that will be used for the study of regularity, is given by $\cL_0 = \cL_0(\s,\t,\l,\L,b)$ with all the linear operators $L$ such that the kernels $K_{e,o}$ are comparable to those the $\s$ fractional Laplacian and some derivation of order $\t$.
\begin{align}
\label{kernel1} (2-\s)\frac{\l}{|y|^{n+\s}}\leq K_e\leq (2-\s)\frac{\L}{|y|^{n+\s}},\\
\label{kernel2} |K_o|\leq (1-\t)\frac{b}{|y|^{n+\t}}.
\end{align}
In order to satisfy the integrability conditions all we need is $\s\in(0,2)$ and $\t\in(0,1)$.

In this family the operators \eqref{moperator1}, \eqref{moperator2} take the explicit form
\begin{align}
\cM^+_{\cL_0}v(x)=\cM^+_\s v(x)+b(1-\t)\int\limits_{\R^n}\frac{|\d_o(v,x,y)|}{|y|^{n+\t}},\\
\cM^-_{\cL_0}v(x)=\cM^-_\s v(x)-b(1-\t)\int\limits_{\R^n}\frac{|\d_o(v,x,y)|}{|y|^{n+\t}},
\end{align}
where $\cM^{\pm}_\s$ are the extremal operators found in \cite{C1}, i.e. 
\begin{align*}
\cM^+_\s v(x)=(2-\s)\int\limits_{\R^n}\frac{\L\d_e^+(v,x;y)-\l\d_e^-(v,x;y)}{|y|^{n+\s}},\\
\cM^-_\s v(x)=(2-\s)\int\limits_{\R^n}\frac{\l\d_e^+(v,x;y)-\L\d_e^-(v,x;y)}{|y|^{n+\s}}.
\end{align*}
$\d^\pm_{e,o}$ denote the positive and negative parts of $\d_{e,o}$ respectively, ($\d_{e,o} = \d_{e,o}^+ - \d_{e,o}^-$).

For ease of notation we also introduce what we call the maximal $\t$ derivative $|D_\t|$, given by
\[
|D_\t|v(x)=(1-\t)\int\limits_{\R^n}\frac{|\d_o(v,x;y)|}{|y|^{n+\t}}dy,
\]
so that we can rewrite the operators as
\[
\cM^{\pm}_{\cL_0}v(x)=\cM^{\pm}_\s v(x)\pm b|D_\t|v(x).
\]

The factors $(2-\s)$ and $(1-\t)$ become important as $\s \to 2$, and $\t \to 1$, as they will allow us to recover second order differential equations with gradient terms as limits of integro differential equations.

Notice that this family admits kernels that could be positive and negative. The natural assumption, due the positivity of the measure in the L\`evy-Khintchine formula, is to consider operators which are elliptic with respect to a family $\cL$ with non negative kernels. Because of this reason we consider also the family $\tilde{\cL}_0\ss\cL_0$, given by all possible operators $L$ with total kernel $K=K_e+K_o\geq0$ satisfying the conditions \eqref{kernel1} and \eqref{kernel2}. We point out that given $v$ smooth, we have the following natural inequalities,
\[
\cM^+_{\cL_0}v(x)\geq\cM^+_{\tilde{\cL}_0}v(x)\geq\cM^-_{\tilde{\cL}_0}v(x)\geq\cM^-_{\cL_0}v(x).
\]
This control will be useful, since we have explicit formulas for the maximal operators in the larger class $\cL_0$.
%%% Ellipticity %%%

\subsection{Ellipticity.}

The reason why we introduce extremal operators is because they are the ones that control \emph{elliptic} non linear operators. Here is the definition of ellipticity for a general family $\cL$ of linear operators.

\begin{definition}\label{ellipticoperator}
Let $\cL$ be a class of linear integro differential operators satisfying \eqref{intconde} and \eqref{intcondo}. We say that a fully non linear operator $I$ is elliptic with respect to the class $\cL$ if
\begin{align}\label{ellipticity}
\cM^-_{\cL}(u-v)(x) \leq Iu(x)-Iv(x)\leq \cM^+_{\cL}(u-v)(x).
\end{align}
\end{definition}

%%% Scaling %%%

\subsection{Scaling.}

A tool we will be using frequently is the scaling. Consider a smooth bounded function $u$ and a operator $I$, elliptic with respect to $\cL \ss \cL_0(\s,\t,\l,\L,b)$, such that
\[
Iu = f \text{ in $\W$}.
\]
If we rescale $u$ by $u_{\a,\b}(x) = \a u(\b x)$ then the equation gets rescaled in the following way,
\[
I_{\a,\b}u_{\a,\b} = f_{\a,\b} \text{ in $\b^{-1}\W$}, 
\]
where
\begin{align*}
(I_{\a,\b}v)(x) &= \a I(\a^{-1}v(\b^{-1}\cdot))(\b x),\\
f_{\a,\b}(x) &= \a f(\b x).
\end{align*}

In particular, if $I=L$ is linear with kernel $K$ then the kernel $K_{\a,\b}$ for $L_{\a,\b}$ gets transformed according to the change of variables formula,
\[
K_{\a,\b}(x,y) = \b^n K(\b x,\b y).
\]

The extremal operators $\cM^\pm_\s$ and $|D_\t|$ scale with order $\s$ and $\t$ respectively, because by the change of variables formula,
\begin{align*}
\cM^\pm_\s u_{\a,\b}(x) &= \a\b^{-\s}(\cM^\pm_\s u)(\b x),\\
|D_\t|u_{\a,\b}(x) &= \a\b^{-\t}(|D_\t|u)(\b x).
\end{align*}

This implies that, going back to $I$ non linear, the operator $I_{\a,\b}$ belongs to some rescaled family of linear operator $\cL_{\a,\b} \ss \cL_0(\s,\t,\b^{-\s}\l,\b^{-\s}\L,\b^{-\t}b)$.

At many points we will use that when $\s>\t$ and $\b$ is small then the rescaled equation is dominated by the derivatives of order $\s$.

%%% Viscosity Solutions %%%

\subsection{Viscosity solutions.}

Viscosity solutions provide the right framework to study fully non linear equations, as seen in the local case in \cite{CC}, and also in the non local case in \cite{BI}.

\begin{definition}\label{viscosity}
A bounded function $u:\R^n\to\R$, upper (lower) semicontinuous in $\bar\Omega$, is said to be a sub solution (super solution) to $Iu=f$, and we write $Iu\geq f$ ($Iu\leq f$), if every time $\varphi$ is a second order polynomial touching $u$ by above (below) at $x$ in a neighborhood $N$, i.e.
\begin{itemize}
\item[(i)] $\varphi(x)=u(x)$,
\item[(ii)] $\varphi(y)>u(y)$ ($\varphi(y)<u(y)$) for every $x\in N\sm\{x\}$,
\end{itemize}
then $Iv(x)\geq f(x)$ ($Iv(x)\leq f(x)$), for $v$ defined as
\begin{align*}
v=\left\{\begin{array}{ll} 
\varphi& \hbox{in} \ N,\\
u&\hbox{in} \ \R^n\sm N.
\end{array}\right. 
\end{align*}
\end{definition}

Later on section \ref{CPE} we will see that in many cases this definition is equivalent to one which includes many more test functions.

%%%%%%%%%%%%%%%%%%%%%%%%%%%%%%%%%%
%%%%%% Statement of Results %%%%%%
%%%%%%%%%%%%%%%%%%%%%%%%%%%%%%%%%%

\section{Statement of Results}\label{SR}

In this section we state the main results obtained in this paper. An important tool used to prove the following theorems is a point estimate, also known as $L^\e$ Lemma. This comes from a partial ABP inequality similar to the one in \cite{C1} and a scaling argument which decreases the effect of the lower order term.

In order to prove our regularity results we will need to impose some assumptions on $\s$ and $\t$. Given $\s_0,\t_0,m,A_0>0$, considered as universal constants, we will assume that the following holds.
\begin{enumerate}
\item[(H1)] $2>\s\geq\s_0>0$, $\min(1,\s)>\t\geq\t_0>0$,
\item[(H2)] $\s-\t \geq m > 0$,
\item[(H3)] $\l A_0(2-\s) \geq b(1-\t)$.
\end{enumerate}

\begin{theorem}\label{ca}
Let $\s_0,\t_0,m,A_0>0$ and assume that H1, H2 and H3 hold. Let $u$ be a bounded function in $\R^n$ such that in $B_1$,
\[
\cM^+_{\tilde{\cL}_0}u \geq -C_0 \ \ \text{and} \ \ \cM^-_{\tilde{\cL}_0}u \leq C_0,
\]
in the viscosity sense. Then there exists a universal exponent $\a>0$ such that $u \in C^\a(B_{1/2})$ and 
\[
\|u\|_{C^\a(B_{1/2})}\leq C(\|u\|_\8 + C_{0})
\]
for some universal constant $C>0$.
\end{theorem}

An immediate corollary is the following.

\begin{corollary}\label{cacor}
Let $\s_0,\t_0,m,A_0>0$ and assume that H1, H2 and H3 hold. Let $I$ be an elliptic operator of the inf-sup type as in \eqref{eqinfsup} with all the linear operators in $\tilde{\cL}_0$ and let $f \in C(\bar B_1)$. Let $u$ be a bounded function in $\R^n$ such that in $B_1$,
\[
Iu = f,
\]
in the viscosity sense. Then there exists a universal exponent $\a>0$ such that $u \in C^\a(B_{1/2})$ and 
\[
\|u\|_{C^\a(B_{1/2})}\leq C(\|u\|_\8 + \|f\|_{\8})
\]
for some universal constant $C>0$.
\end{corollary}

Coming back to Theorem \ref{ca}, we would like to point out that our bounds remain uniform as $\s\to 2$ and $\t\to 1$, which allows us to recover H\"older regularity for equations with bounded measurable coefficients including gradient terms. For fixed $\s$ and $\t$ these results were proven in \cite{S1} and \cite{BK} by using analytic techniques. These estimates are not uniform in $\s$ and blow up as the order goes to the classical one.

The order $\a$ of our H\"older estimates deteriorates as $\t\to\s$. In this critical case $\s=\t$, both terms in the equation are of the same order and rescaling the equation doesn't have any effect on the $\t$ derivative, hence our argument doesn't work. It is known from the previous work in \cite{S1} and \cite{BK} that the same result holds even when $\s=\t$. By combining both results, we can get regularity uniformly in $\s$ and $\t$, disregarding the separation between $\s$ and $\t$ (hypothesis H2).

To get higher regularity we will need to add an extra assumption to the kernels, which is a modulus of continuity of $K_e$ and $K_o$ in measure. More precisely, given $\r_0>0$, we define the class $\cL_1 = \cL_1(\s,\t,\l,\L,b,\rho_0,C) \ss \tilde\cL_0(\s,\t,\l,\L,b)$ such that it contains all the linear operators $L$ with kernels $K = K_e+K_o \geq 0$ such that $K_e$ and $K_o$ satisfy \eqref{kernel1} and \eqref{kernel2} respectively and
\begin{align}\label{L1}
\int\limits_{\R^n\sm B_{\r_0}}\frac{|K(y)-K(y-h)|}{|h|}dy\leq C
\end{align}
for every $|h|\leq \r_0/2$.

A sufficient condition for \eqref{L1} is for example that $|\n K(y)|\leq\L/|y|^{n+1+\s}$.

In this smaller class we are able to get $C^{1,\a}$ by studying the incremental quotients of solutions and using the a priori $C^\a$ estimates given by Theorem \ref{ca}. The proof follows the ideas of \cite{CC} and \cite{C1}.

\begin{theorem}\label{c1a}
Let $\s_0,\t_0,m,A_0>0$ and assume that H1, H2 and H3 holds. Let $I$ be an elliptic operator of the inf-sup type as in \eqref{eqinfsup} with all the linear operators in $\cL_1$. There is $\r_0>0$ small enough so that if $u$ is a bounded function in $\R^n$ such that in $B_1$,
\[
Iu = 0,
\]
in the viscosity sense. Then there is a universal $\a>0$ such that $u \in C^{1,\a}(B_{1/2})$ and
\[
\|u\|_{C^{1,\a}(B_{1/2})} \leq C\|u\|_\8
\]
for some universal $C > 0$.
\end{theorem}

In the proofs of our regularity results the odd part doesn't have to be of a fixed order. We could ask for example
\[
 |K_o|\leq b\max\1\frac{1-\t_1}{|y|^{n+\t_1}},\frac{1-\t_2}{|y|^{n+\t_2}}\2
\]
with $0<\t_1\leq \t_2<\min(1,\s)$. The reason is that the proofs will treat the lower order term as a perturbation term that can be made small enough after a dilation large enough. For the sake of keeping the exposition simpler we decided to restrict to the case of $\t_1=\t_2=\t$.

%%%%%%%%%%%%%%%%%%%%%%%%%%%%%%%%%%%%
%%%%%% Qualitative properties %%%%%%
%%%%%%%%%%%%%%%%%%%%%%%%%%%%%%%%%%%%

\section{Qualitative properties}\label{CPE}

This section is devoted to prove basic results that concern the definition of viscosity solution. First we take a look to the monotonicity properties which are inherited from assuming that the operator $I$ is elliptic with respect to a family $\cL$ with non negative kernels. Second we see how the set of test functions can be enlarged in the definition of viscosity solutions. We use these tools to prove the stability, comparison and maximum principle and existence of solutions for the Dirichlet problem.

%%% Monotonicity %%%

\subsection{Monotonicity.}

\begin{lemma}[Monotonicity]\label{mono}
 Let $I$ be a elliptic with respect to a family $\cL$ of linear operators with non negative kernels. Let $u$ and $v$ be two bounded functions in $C^{1,1}(x)$ such that $v\geq u$ and $v(x_0)=u(x_0)$, then
\[
 Iv(x_0)\geq Iu(x_0).
\]
\end{lemma}

\begin{proof}
By the ellipticity,
\[
 Iv(x_0) - Iu(x_0) \geq \cM^-_\cL(v-u)(x_0).
\]
Let $w(y) = (v-u)(x_0+y)$ such that $w(y)\geq0$ with equality for $y=0$. Then for any $L \in \cL$ with kernel $K\geq0$ we have
\[
Lw(0) = P.V. \int w(y)K(y) \geq 0.
\]
By taking the infimum we get that $\cM^-_\cL w(x_0) \geq 0$ which concludes the proof.
\end{proof}

\begin{lemma}\label{minlema}
Let $I$ be an elliptic operator with respect to a class $\cL$ of non negative kernels. Let $u$, $v$ be viscosity solutions of $Iu \leq f$, then $w=\min(u,v)$ is also a super solution.
\end{lemma}

\begin{proof}
Let $\varphi$ be a function touching $w$ by below at $x$ in $N$ and assume without loss of generality that $w(x_0) = u(x_0)$. Then $\varphi$ also touches $u$ by below at $x_0$ in $N$ and we use its equation. For
\begin{align*}
v=\left\{\begin{array}{ll} 
\varphi& \hbox{in} \ N,\\
u&\hbox{in} \ \R^n\sm N,
\end{array}\right.
\end{align*}
we have $Iv(x_0) \leq f(x_0)$.

Let $\tilde v$ be defined by
\begin{align*}
\tilde v=\left\{\begin{array}{ll} 
\varphi& \hbox{in} \ N,\\
w&\hbox{in} \ \R^n\sm N.
\end{array}\right.
\end{align*}
Then by the monotonicity Lemma \ref{mono} applied to $v$ and $\tilde v$ at $x_0$ we get $I\tilde v(x_0) \leq Iv(x_0) \leq f(x_0)$ which concludes that $Iw \leq f$.
\end{proof}

%%% A larger class %%%

\subsection{A larger class of test functions.}

\begin{lemma}\label{testfunc}
Let $I$ be elliptic with respect to a class $\cL$ of non negative kernels satisfying \eqref{intconde} and \eqref{intcondo} uniformly. Let $u:\R^n \to \R$ such that $Iu \geq f$ in the viscosity sense and $\varphi$ touching $u$ by above at $x$ in a neighborhood $N$. Then $Iv(x) \geq f(x)$ for $v$ defined as
\begin{align*}
v=\left\{\begin{array}{ll} 
\varphi& \hbox{in} \ N,\\
u&\hbox{in} \ \R^n\sm N.
\end{array}\right. 
\end{align*}
given that $\varphi \in C^{1,1}(x)$.
\end{lemma}

\begin{proof}
Fix $p$ and $q$ second order polynomials that touch $\varphi$, by below and above respectively, at $x$ in $B_r(x) \ss N$. Let
\begin{align*}
w=\left\{\begin{array}{ll} 
q& \hbox{in} \ B_r(x),\\
u&\hbox{in} \ \R^n\sm B_r(x),
\end{array}\right.
\quad v_r=\left\{\begin{array}{ll} 
q& \hbox{in} \ B_r(x),\\
v&\hbox{in} \ \R^n\sm B_r(x).
\end{array}\right. 
\end{align*}
By the ellipticity
\[
Iv(x) \geq Iv_r(x) + \cM^-_\cL (v-v_r)(x),
\]
and thanks to the monotonicity Lemma \ref{mono} applied to $v_r\geq w$ we have $Iv_r(x)\geq Iw(x)$, so that
\[
Iv(x)\geq Iw(x) + \cM^-_\cL (v-v_r)(x).
\]
Note that $Iw(x) \geq f(x)$, so we only need to estimate the second term. Now, $v-v_r$ is supported in $B_r(x)$ and it is equal to $\varphi-q$ which is bounded by $-(q-p)$ and zero. For $L\in\cL$ with kernel $K$,
\begin{align*}
L(v-v_r)(x) &= \int_{B_r} \d_e((v-v_r),x;y)K_e(y)dy + \int_{B_r} \d_o((v-v_r),x;y)K_o(y)dy\\
&\geq -C\3\int_{B_r} |y|^2K_e(y)dy + \int_{B_r} |y|K_o(y)dy\4\\
&\geq -C\e,
\end{align*}
for $\e>0$ arbitrarily small if $r=r(\e)$ is small enough (independent of $L\in\cL$). By taking the infimum above among every $L\in\cL$ we get that $Iv(x) \geq f(x) - C\e$ and we just need to take $\e\to0$ to conclude.
\end{proof}

Next we have an even stronger result, that tells us that we can compute $I$ classically every time we have a $\varphi \in C^{1,1}(x)$ touching by below.

\begin{lemma}\label{welldefined}
Let $I$ be an elliptic operator of the inf-sup (sup-inf) type as in \eqref{eqinfsup} with all the linear operators in $\tilde{\cL}_0$ satisfying H1. Let $u:\R^n\to\R$ such that $Iu \leq f$  in the viscosity sense and $\varphi$ touching $u$ by below at $x$ in a neighborhood $N$. Then $Iu(x)$ is defined in the classical sense and we have $Iu(x)\leq f(x)$ given that $\varphi \in C^{1,1}(x)$.
\end{lemma}

To prove Lemma \ref{welldefined} we need an interpolation result that will allow us to replace the $\t$ derivative by the $\s$ derivative and a residue term evaluated at the test function $\varphi$. This result is also useful when the function touching by below is the convex envelope as $\d_e^-(\varphi) = 0$.

\begin{lemma}\label{interpolation}
Let $u:\R^n \to \R$, $x\in\R^n$, $2>\s>\t>0$ and $r_0>0$ such that the following integrals are finite,
\[
\int_{B_{r_0}} \frac{\d^+_e(u,x;y)}{|y|^{n+\s}}dy \ \ \text{and} \ \ \int_{B_{r_0}}\frac{|\d_o(u,x;y)|}{|y|^{n+\t}} dy.
\]
Let $\varphi$ be a function defined in $B_{r_0}(x)$ and touching $u$ by below at $x$. Then
\begin{align*}
&\int_{B_{r_0}} \l(2-\s)\frac{\d^+_e(u,x;y)}{|y|^{n+\s}} - b(1-\t)\frac{|\d_o(u,x;y)|}{|y|^{n+\t}} dy \geq\\
&\int_{B_{r_0}} \a\l(2-\s)\frac{\d^+_e(u,x;y)}{|y|^{n+\s}} - b(1-\t)\frac{\d_e^-(\varphi,x;y)+|\d_o(\varphi,x;y)|}{|y|^{n+\t}}dy,
\end{align*}
for $\a \in (0,1)$ given that
\begin{align}\label{interpol}
r_0 \leq \1\frac{(1-\a)\l(2-\s)}{b(1-\t)}\2^{1/(\s-\t)}.
\end{align}
\end{lemma}

\begin{proof}
Since $\varphi$ touches $u$ by below, we have that for every $y \in B_r$,
\begin{align*}
\d_e^+(u - \varphi,x;y) &= (u-\varphi)(x+y)+(u-\varphi)(x-y),\\
&\geq |(u-\varphi)(x+y)-(u-\varphi)(x-y)|,\\
&= |\d_o(u-\varphi,x;y)|,
\end{align*}
and also,
\begin{align*}
\d_e^+(u,x;y) &\geq \d_e^+(u - \varphi,x;y) - \d_e^-(\varphi,x;y),\\ 
|\d_o(u-\varphi,x;y)| &\geq |\d_o(u,x;y)| - |\d_o(\varphi,x;y)|,
\end{align*}
so that 
\[
\d_e^+(u,x;y) - |\d_o(u,x;y)| \geq -\d_e^-(\varphi,x;y)-|\d_o(\varphi,x;y)|.
\]
Now we can replace $|\d_o|$ by $\d_e^+$ in the integral, 
\begin{align*}
&\int_{B_{r_0}} \l(2-\s)\frac{\d^+_e(u)}{|y|^{n+\s}} - b(1-\t)\frac{|\d_o(u)|}{|y|^{n+\t}} dy \geq\\
&\int_{B_{r_0}}\d_e^+(u)\3 \frac{\l(2-\s)}{|y|^{n+\s}} - \frac{b(1-\t)}{|y|^{n+\t}}\4 - b(1-\t)\frac{\d_e^-(\varphi)+|\d_o(\varphi)|}{|y|^{n+\t}}dy.
\end{align*}

By using that
\[
r_0 \leq \1\frac{(1-\a)\l(2-\s)}{b(1-\t)}\2^{1/(\s-\t)},
\]
and that $\s > \t$ we can substitute the difference of the fractions by $\a$ times $|y|^{-(n+\s)}$,
\[
\int_{B_{r_0}}\d_e^+(u)\3 \frac{\l(2-\s)}{|y|^{n+\s}} - \frac{b(1-\t)}{|y|^{n+\t}}\4dy \geq \a\l(2-\s)\int_{B_{r_0}}\frac{\d_e^+(u)}{|y|^{n+\s}}.
\]
\end{proof} 

\begin{proof}[Proof of Lemma \ref{welldefined}]
We check first that $Lu$ can be computed in the classical sense at $x$. Because $u$ is bounded we only  care about the convergence of the integrals around the origin.

Let $\varphi$ be defined in $B_{r_0}(x)$ and for $r\leq r_0$
\[
v_r(y) =
\begin{cases}
 u        \text{ in } B_r(x),\\
 \varphi  \text{ in } \R^n\sm B_r.
\end{cases}
\]

The differences $\d^-_e(v_r,x;y)$, parametrized by $r$, decrease to $\d^-_e(u,x;y)$ as $r$ goes to zero. Since
\[
\int_{B_{r_0}} \frac{\d_e^-(v_{r_0},x;y)}{|y|^{n+\s}}dy <\8,
\]
we have by monotone convergence that
\[
\int_{B_{r_0}} \frac{\d_e^-(u,x;y)}{|y|^{n+\s}}dy <\8.
\]

By using $f(x) \geq \cM^-_{\tilde{\cL}_0}v_r(x)$, which implies $f(x) \geq \cM^-_{\cL_0}v_r(x)$, and the boundedness of $u$,
\[
M \geq \int_{B_{r_0}} \l(2-\s)\frac{\d_e^+(v_r,x;y)}{|y|^{n+\s}} - b(1-\t)\frac{|\d_o(v_r,x;y)|}{|y|^{n+\t}}dy
\]
for some $M$ independent of $r$. Use now Lemma \ref{interpolation} to keep only the term with $\d^+_e(v_r)$, this requires $r_0$ sufficiently small,
\begin{align*}
&M + b(1-\t)\int_{B_{r_0}} \frac{|\d_o(\varphi,x;y)| + \d_e^-(\varphi,x;y)}{|y|^{n+\t}}dy \geq \\ 
&\frac{\l(2-\s)}{2}\int_{B_{r_0}} \frac{\d_e^+(v_r,x;y)}{|y|^{n+\s}} dy.
\end{align*}
The left hand side above is finite and independent of $r$. By Fatou's Lemma,
\[
\int_{B_{r_0}} \frac{\d_e^+(u,x;y)}{|y|^{n+\s}}dy <\8.
\]

Now recall from the proof of Lemma \ref{interpolation} the identity
\[
|\d_o(v_r,x;y)| \leq |\d_o(\varphi,x;y)| + \d_e^+(v_r,x;y) + \d_e^-(\varphi,x;y).
\]
The last two terms are integrable against $|y|^{-(n+\s)}$ around the origin and therefore they are also integrable against $|y|^{-(n+\t)}$ around the origin as well as the whole right hand side. Moreover the integral can be bounded by above independently of $r$. By Fatou's Lemma we then get that $|\d_o(u,x;y)|$ is integrable against $|y|^{-(n+\t)}$ in $B_{r_0}$.

We have shown that each term $\d_e(u,x;y)/|y|^{n+\s}$ and $|\d_o(u,x;y)|/|y|^{n+\t}$ is integrable, then for every linear operator $L_{\a,\b}u(x)$ is well defined. Therefore $Iu(x)$ can be computed by being an inf-sup combination of $L_{\a,\b}$. To see that $Iu(x) \leq f(x)$ we use the ellipticity,
\begin{align*}
Iu(x) &\leq Iv_r(x) + \cM_{\cL}^+(u-v_r)\\
      &\leq f(x) + \cM_\s^+(u-v_r)(x) + b|D_\t|(u-v_r)(x).
\end{align*}
Both integrals go to zero by absolute continuity.
\end{proof}

%%% Stability %%%

\subsection{Stability.}

We are interested in studying limit of sub or super solutions. To state the result we need first to recall the definition of $\G$ convergence.

\begin{definition}
We say that a sequence of lower semicontinuous functions $u_k$ $\G$-converge to $u$ in a set $\W$ if the two following conditions hold
\begin{itemize}
\item[(i)] For every sequence $x_k\to x$ in $\Omega$, $\liminf_{\k\to\infty}u_k(x_k)\geq u(x)$.
\item[(ii)] For every $x\in\Omega$, there is a sequence $x_k\to x$ in $\Omega$ such that
\[
\limsup\limits_{k\to\infty}u_k(x_k)=u(x).
\]  
\end{itemize}
\end{definition}

\begin{lemma}\label{converge}
Let $I$ be an elliptic operator with respect to a class $\cL$ with non negative kernels and satisfying the integrability conditions \eqref{intconde} and \eqref{intcondo} uniformly. Let $u_k$ be a sequence of functions that are uniformly bounded in $\R^n$ and lower semicontinuous in $\W\ss\R^n$ such that
\begin{itemize}
\item[(i)] $Iu_k\leq f_k$ in $\Omega$
\item[(ii)] $u_k\to u$ in the $\G$ sense in $\Omega$,
\item[(iii)] $u_k\to u$ a.e. in $\R^n$ and
\item[(iv)] $f_k\to f$ locally uniformly un $\Omega$ for some continuous function $f$.
\end{itemize}
Then $Iu\leq f$ in $\Omega$.
\end{lemma}

\begin{proof}
Let $\varphi$ be a test function touching $u$ by below at $x$ in $N$. Because $u_k - \varphi$ $\G$-converges to $(u-\varphi)$ there exists a sequence $x_k \to x$ such that,
\[
(u_k -\varphi)(x_k) = \inf_N (u_k-\varphi) = d_k.
\]
Therefore $\varphi+d_k$ touches $u_k$ at $x_k$ in $N$, starting at some $k$ sufficiently large.

Let
\begin{align*}
v_k = \begin{cases}
\varphi + d_k &\text{ in } B_r(x),\\
u_k &\text{ in } \R^n \sm B_r(x),
\end{cases} \quad 
v = \begin{cases}
\varphi  &\text{ in } B_r(x),\\
u &\text{ in } \R^n \sm B_r(x).
\end{cases}
\end{align*}
By using the equation we know that $Iv_k(x_k)\leq f_k(x_k)$.

For $z \in B_{r/2}(x)$ we have by ellipticity,
\begin{align*}
|Iv_k(z) - Iv(z)| &\leq \max(|\cM^+_\cL(v_k-v)|,|\cM^+_\cL(v_k-v)|),\\
&\leq \sup_{L\in\cL} |L(v_k-v)(z)|.
\end{align*}
For a given $L\in\cL$ with kernel $K$ we have,
\begin{align*}
|L(v_k-v)(z)| \leq \int_{\R^n\sm B_{r/2}}|(v_k-v)(x+y)|K(y)dy.
\end{align*}
The integrand goes to zero a.e. when $k\to\8$ and it is dominated by
\[
(\|v_k\|_\8 + \|v\|_\8)K(y)\chi_{\R^n\sm B_{r/2}} \in L^1.
\]
Then by dominated convergence $|L(v_k-v)(z)| \to 0$ as $k\to\8$ uniformly in $z \in B_{r/2}(x)$ and $L \in \cL$. This implies that $|Iv_k(z) - Iv(z)|$ also goes to zero uniformly in $z \in B_{r/2}(x)$.

Finally, using that $Iv$ is continuous in $B_{r/2}(x)$,
\begin{align*}
|Iv_k(x_k) - Iv(x)| &\leq |Iv_k(x_k) - Iv(x_k)| + |Iv(x_k) - Iv(x)| \to 0.
\end{align*}
Finallt we have
\begin{align*}
Iv(x) &\leq Iv_k(x_k) + |Iv_k(x_k) - Iv(x)|,\\
&\leq f(x_k) + |Iv_k(x_k) - Iv(x)|,\\
&\leq f(x) + |Iv_k(x_k) - Iv(x)| + |f_k(x_k) - f(x)|.\\
\end{align*}
Take $k\to\8$ and use also that $f_k\to f$ locally uniformly to conclude. 
\end{proof}

%%% Comparison and maximum principle %%%

\subsection{Comparison and maximum principle for viscosity solutions.}

Lemma \ref{difofsol} says that the difference of two viscosity solutions is the solution of an equation in the same ellipticity class. Theorem \ref{comparisonprinciple} is the comparison principle which implies in particular the maximum principle for sub solution. Instead of having to prove an ABP type result, as it is used in chapter 5 of \cite{CC}, we take advantage of Lemma \ref{welldefined} in order to evaluate the operators in the classical sense whenever is needed.

\begin{lemma}\label{difofsol}
Let $I$ be an elliptic operator of the inf-sup type as in \eqref{eqinfsup} with all the linear operators in $\tilde\cL_0$ satisfying H1. Let $u$ and $v$ two bounded functions such that $Iu \geq f$ and $Iv \leq g$ in the viscosity sense in $\W$. Then $\cM^+_{\tilde{\cL}_0}(u-v) \geq f-g$ in the viscosity sense in $\W$.
\end{lemma}

The proof is straightforward when either $u$ or $v$ is smooth because of the non negativity of the kernels. In the general case we proceed by regularizing the functions by their inf or sup convolutions.

\begin{definition}
Given an lower (upper) semi continuous function $u$ and a parameter $\e>0$ the inf (sup) convolution $u_\e$ ($u^\e$) is given by
\[
u_\e(x) = \inf_y u(x+y) + \frac{|y|^2}{\e} \quad \1u^\e(x) = \sup_y u(x+y) - \frac{|y|^2}{\e}\2.
\]
\end{definition}

The proof of the following property con be found for instance in the beginning of chapter 5 in \cite{CC}.

\begin{lemma}
If $u$ is bounded and lower semicontinuous in $\R^n$ then $u_\e$ $\G$-converges to $u$.
\end{lemma}

\begin{lemma}\label{eqnforsup}
If $f$ is a continuous function and $I$ is elliptic with respect to a class $\cL$ with non negative kernels. Then if $Iu \leq f$ in the viscosity sense, $Iu_\e \leq f - d_\e$ also in the viscosity sense, where $d_\e \to 0$ as $\e \to 0$ depending only on the modulus of continuity $\rho$ of $f$.
\end{lemma}

\begin{proof}
Let $\varphi$ be a test function that touches $u_\e$ by below at $x$ in $N$.

For $\e$ sufficiently small, there is some $(x+h) \in N$ such that $u_\e(y) \leq u(y+h) + |h|^2/\e$ with equality at $y=x$. (See the beginning of chapter 5 in \cite{CC}).

Then $\varphi - |h|^2/\e$ touches $u$ at $x+h$ in $N$ and $Iv(x+h) \leq f(x+h) \leq f(x) + \rho(|h|)$ for
\[
v = \begin{cases}
\varphi - \frac{|h|^2}{\e} &\text{ in } B_{r/2}(x+h),\\
u &\text{ in } \R^n \sm B_{r/2}(x+h).
\end{cases}
\]

By ellipticity the value of $Iv$ does not change by adding a constant, $I(v+ |h|^2/\e)(x) \leq f(x) + \rho(|h|)$. Then by the monotonicity Lemma \ref{mono} we also have that $Iw(x) \leq f(x) + \rho(|h|)$ for
\[
w = \begin{cases}
\varphi &\text{ in } N,\\
u_\e  &\text{ in } \R^n \sm N,
\end{cases}
\]
because as we already noticed $u_\e(y) \leq u(y+h) + |h|^2/\e$. This concludes the proof.
\end{proof}

\begin{proof}[Proof of Lemma \ref{difofsol}]
Assume first that $u$ is upper semicontinuous in $\R^n$ and $v$ is lower semicontinuous in $\R^n$.

Thanks to Lemma \ref{eqnforsup}, we have that $Iu^\e\geq f-d_\e$ and $Iv_\e\leq f+d_\e$ with $-u^\e\to-u$ and $v_\e\to v$ in the $\G$-sense and $d_\e\to0$. By the stability of viscosity solutions, Lemma \ref{converge}, we just need to prove that $\cM^{+}_{\tilde{\cL}_0}(u^\e-v_\e)\geq f-g-2d_\e$ in $\W$ in the viscosity sense.

Let $\varphi$ be a test function touching $u^\e-v_\e$ from above at a point $x$. For any $\e>0$, $u^\e$, $-v_\e$ and $u^\e-v_\e$ are semiconvex functions, hence there is a paraboloid for each of them touching then from below at $x$. If $\varphi \in C^{1,1}(x)$ then both $u^\e$ and $v_\e$ are also $C^{1,1}(x)$. By Lemma \ref{welldefined} we can evaluate $Iu^\e(x)$, $Iv_\e(x)$ and $\cM^+_{\tilde{\cL}_0}(u^\e-v_\e)(x)$ in the classical sense and they satisfy
\[
\cM^+_{\tilde{\cL}_0}(u^\e-v_\e)(x)\geq Iu^\e(x)-Iv_\e(x)\geq f(x)-g(x)-2d_\e.
\]
Since $\varphi$ touches $u^\e-v_\e$ from above at $x$
\[
\cM^+_{\tilde{\cL}_0}\varphi(x)\geq f(x)-g(x)-2d_\e.
\]
This says that $\cM^+_{\tilde{\cL}_0}(u^\e-v_\e)\geq f-g-2d_\e$ in the viscosity sense and completes the proof under the semicontinuity assumptions in $\R^n$.

Now we will not assume the lower and upper semincontinuity outside of $\bar\W$. There are sequences $u_k$ and $v_k$, upper and lower semicontinuous respectively such that
\begin{itemize}
\item[(i)] $u_k=u$ and $v_k=v$ in $\bar\Omega$ for every $k$,
\item[(ii)] $u_k\to u$ and $v_k\to v$ a.e. in $\R^n\sm\bar\Omega$,
\item[(iii)] $Iu_k\geq f_k$ and $Iv_k\leq g_k$, with $f_k\to f$, $g_k\to g$ locally uniformly in $\Omega$.
\end{itemize}
By having such sequences we just have to apply the first part of this proof and the stability, Lemma \ref{converge}, to conclude the proof.

We can construct the sequences satisfying the first two items above by doing a standard mollification of $u$ and $v$ away from $\W$ and then filling the gap in a semicontinuous way. The function $u_k-u$ vanishes in $\Omega$, hence $\cM^-_{\tilde{\cL}_0}(u_k-u)$ is defined in the classical sense in $\W$ and
\begin{align*}
\cM^-_{\tilde{\cL}_0}(u_k-u)(x) &\geq -\int_{\R^n\sm B_{dist(x,\partial \Omega)}}|u_k(x+y)-u(x)|K(y)dy\\
&= h_k(x),
\end{align*}
where $K=\L\frac{2-\s}{|y|^{n+\s}}+b\frac{1-\t}{|y|^{n+\t}}$.	

The functions $h_k(x)$ are continuous in $\W$ and by dominated convergence $h_k\to 0$ locally uniformly in $\W$ as $k\to0$. Let $\varphi \in C^{1,1}(x)$ touching $u_k$ from above at $x$ in $N$ and $v_k$ defined by
\begin{align*}
v_k = \begin{cases}
\varphi &\text{ in $N$,}\\
u_k &\text{ in $\R^n\sm N$.}
\end{cases}
\end{align*}
The functions $v_k + u - u_k$ are also in $C^{1,1}(x)$ and touch $u$ by above at $x$. By Lemma \ref{testfunc} we have that $I(v_k + u - u_k)(x) \geq f(x)$. By ellipticity
\begin{align*}
Iv_k(x) &\geq I(v_k + u - u_k)(x) + \cM^-_{\tilde{\cL}_0}(u-u_k)(x),\\
&\geq f(x) + h_k(x).
\end{align*}
So we have that (iii) above is also satisfied.
\end{proof}

\begin{theorem}[Comparison Principle]\label{comparisonprinciple}
Let $I$ be an elliptic operator of the inf-sup type as in \eqref{eqinfsup} with all the linear operators in $\tilde\cL_0$ satisfying H1. Let $\W$ be a bounded open set and $u$, $v$ two bounded functions such that
\begin{itemize}
\item[(i)] $Iu\geq f$ and $Iv\leq f$ in $\W$ in the viscosity sense for some $f \in C(\W)$,
\item[(ii)] $u\leq v$ in $\R^n\sm\W$.
\end{itemize}
Then $u\leq v$ in $\Omega$.
\end{theorem}

Here, as in \cite{C1}, the proof is also based on using a barrier function as
\[
 \varphi(x) = \min(1,|x|^2/4).
\]

\begin{lemma}
Let $s \in (0,1)$ and $\varphi_s = \varphi(sx)$ for $\varphi$ defined above. There exists $\d>0$ and some $s$ small enough such that,
\[
\cM^-_{\cL_0}\varphi_s \geq \d \text{ in $B_1$}.
\]
\end{lemma}

\begin{proof}
First take $s=1$ to get that
\[
 \cM^-_\s \varphi \geq \d_1 \text{ in $B_1$}.
\]
This inequality comes from the fact that $\d_e(\varphi,x;y) \geq \d_2$ for $x\in B_1$ and every $y$. In fact, if $x\pm y$ are both in $B_2$ or both outside $B_2$ it is immediate. If only $x+y$ is in $B_2$ then we use that, $\varphi(x) \leq 1/4$ for $x\in B_1$,
\[
\d_e(\varphi,x;y) = 1 + \varphi(x+y) - 2\varphi(x) \geq 1/2.
\]

On the other hand, since $\varphi$ is smooth we have that $|D_\t|\varphi \leq \d_3$ in $B_1$, for some finite $\d_3>0$. Now recall scaling properties from Section \ref{VSP}. We have that
\begin{align*}
 \cM_\s^-\varphi_s(x) &= s^{-\s}(\cM_\s^-\varphi)(sx) \geq s^{-\s}\d_1,\\
 |D_\t|\varphi_s(x) &= s^{-\t}(|D_\t|\varphi)(sx) \leq s^{-\t}\d_3.
\end{align*}
Which implies
\begin{align*}
 \cM_{\cL_0}^-\varphi_s &\geq s^{-\s}(\d_1 - s^{\t-\s}\d_3),\\
&\geq \d.
\end{align*}
For $s$ small enough.
\end{proof}

\begin{proof}[Proof of Theorem \ref{comparisonprinciple}]
By Lemma \ref{difofsol} we know that for $w=u-v$, $\cM^+_{\tilde\cL_0}w \geq 0$ in the viscosity sense in $\W$. We will proof from here that $\sup_{\W} w \leq \sup_{\R^n\sm\W}w := M$.

Let $\W \ss B_R$ ($R\geq 1$) and take $\psi(x) = \varphi_s(x/R)$ for $\varphi_s$ as in the previous lemma. Since $\tilde\cL_0 \ss \cL_0$,
\begin{align*}
\cM^-_{\tilde\cL_0}\psi &\geq \cM^-_{\cL_0}\psi,\\
 &\geq R^{-\s}\1\cM^-_{\cL_0}\varphi_s\2(x/R),\\
 &\geq R^{-\s}\d,
\end{align*}
in $\W$. Fix $\e>0$ and consider
\[
 \psi_\e(x) = M + \e(1- \psi(x)),
\]
which satisfies $\cM^+_{\tilde\cL_0}\psi_\e \leq -\e R^{-\s}\d < 0$ in $\W$.

If $\inf (\psi_\e-w)<0$ then there is some translation $\psi_\e+d$ such that $\psi_\e+d$ touches $w$ by above in $x\in\W$. This can not happen because of Lemma \ref{testfunc} which says that in that case $\cM_{\tilde\cL_0}^+ (\psi_\e+d)(x) = \cM_{\tilde\cL_0}^+ \psi_\e(x)\geq 0$. Therefore $\psi_\e \geq w$ and by letting $\e\to0$ we get to the conclusion of the Theorem.
\end{proof}

%%% Existence %%%

\subsection{Existence of solutions for the Dirichlet problem.}

\begin{theorem}\label{exist}
Let $I$ be an elliptic operator of the inf-sup type as in \eqref{eqinfsup} with all the linear operators in $\tilde\cL_0$ satisfying H1. Let $\W\subset\R^n$ be an open bounded set satisfying the exterior ball condition. Let $g:\R^n\sm\W\to\R$ be a function which is globally bounded and continuous on $\p\W$. Then there exist a viscosity solution $u \in C(\bar\W)$ of 
\begin{align*}
Iu(x)=0, \text{ in } \W,\\
u=g, \text{ in } \R^n\sm\W.
\end{align*}
\end{theorem}

The proof is based on the Perron's method. The first two lemmas account to the construction of a solution and the third one regards with achieving the boundary data.

\begin{lemma}
Let $I$ be an elliptic operator with respect to a class $\cL$ with non negative kernels and satisfying the integrability conditions \eqref{intconde} and \eqref{intcondo} uniformly. Let $S$ a set of viscosity solutions of $Iv \geq 0$ in $\W$. Then $\bar u$, the upper semicontinuous envelope in $\W$ of the function $u$ defined by
\[
 u(x) = \sup_{v\in S}v(x),
\]
also satisfies $I\bar u \geq 0$ in $\W$ in the viscosity sense.
\end{lemma}

\begin{proof}
Let $\varphi$ be a test function touching $\bar u$ by above at $x \in \W$ in a neighborhood $N$.

The fact that $\bar u$ is defined as the upper semicontinuous envelope of $u(y) = sup_{v\in S} v(y)$ implies that for our given $x$ there exist a sequence $\{(x_k,v_k)\} \ss (N \cap \W) \times S$ such that
\begin{enumerate}
 \item[(i)] $(x_k,v(x_k)) \to (x,\bar u(x))$ as $k\to \8$,
 \item[(ii)] For any $y_k \to x$ we have that $\liminf_{k\to\8} v_k(y_k) \geq \bar u(x)$.
\end{enumerate}
These are the two sufficient conditions to prove the stability Lemma \ref{converge}. The same proof applies here to show that $I\varphi(x) \geq 0$ and then conclude that $Iu \geq 0$ in the viscosity sense.
\end{proof}

\begin{lemma}
Let $I$ be an elliptic operator with respect to a class $\cL$ with non negative kernels and satisfying the integrability conditions \eqref{intconde} and \eqref{intcondo} uniformly. Let $u$ be a viscosity subsolution of $Iu \geq 0$ in $\W$ such that $\underbar u$, its lower semicontinuous envelope, it is not a viscosity supersolution of $Iu \leq 0$. Then there is function $U$ such that
\begin{enumerate}
 \item[(i)] $U$ is a viscosity subsolution of $IU\geq 0$ in $\W$.
 \item[(ii)] $U=u$ in $\R^n \sm \W$.
 \item[(iii)] $\displaystyle\sup_{x\in\W} (U-u)(x) > 0$.
\end{enumerate}
\end{lemma}

\begin{proof}
 Let $\varphi$ be a test function touching $\it{\underbar u}$ by below in $x_0 \in B_{r_0}(x_0) \ss\W$ such that $Iv(x_0) > 0$ for,
\[
 v = \begin{cases}
      \varphi &\text{ in $B_{r_0}(x_0)$,}\\
      \underbar u &\text{ in $\R^n \sm B_{r_0}(x_0)$.}
     \end{cases}
\]
By continuity we also have that $Iv > \d > 0$ in $B_{r_1}(x_0) \ss B_{r_0/2}(x_0)$.

Let $\e_1,\e_2>0$ to be fixed and
\[
 \psi(y) = \varphi(y) - \e_1|y-x_0|^2 + \e_2.
\]
We have that $\psi \leq u$ in $\R^n \sm B_{r_1}(x_0) \ss \R^n \sm \W$ if $\e_2 < r_1^2$. We want to choose $\e_1$ and $\e_2$ such that $U = \min(\psi,u)$ satisfies $IU \geq 0$.

Let $\eta$ be a test function touching $U$ by above at $x_1\in\W$ in a neighborhood $N$. If $U(x_1) = u(x_1)$ then the inequality follows from the Lemma \ref{testfunc}. If $U(x_1) = \psi(x_1) > u(x_1)$ then necessarily $x_1\in B_{r_1}(x_0)$. By the lower semicontinuity, $\psi > u$ in some open neighborhood around $x_1$ and contained in $B_{r_1}(x_0)$. Because $\psi$ is smooth, $IU(x_0)$ is classically defined and we just have to check that it is non negative. 

Let
\[
 w = \begin{cases}
      \psi &\text{ in $B_{r_0}(x_0)$,}\\
      \underbar u &\text{ in $\R^n \sm B_{r_0}(x_0)$.}
     \end{cases}
\]
By monotonicity and ellipticity,
\begin{align*}
 IU(x_1) &\geq Iw(x_1),\\
&\geq Iv(x_1) + \cM^-_\cL (w-v)(x_1),\\
&\geq \d + \inf_{L \in \cL}L(w-v)(x_1).
\end{align*}
Notice that $w-v = \e_2 - \e_1|y-x_0|^2$ in $B_{r_0}(x_0)$ and it is zero outside. Recall that $x_1\in B_{r_0/2}(x_0)$, so for any $L \in \cL$ with kernel $K$, 
\begin{align*}
 L(w-v)(x_1) &= \int\d(w-v,x_1;y)K(y)dy,\\
&\geq -\e_1\int_{B_{2r_0}}\d(|\cdot-x_0|^2,x_1;y)K(y)dy,\\ 
{} &-\min\1\e_1r_0^2,\e_2\2\int_{\R^n \sm B_{r_0/2}}K(y)dy.
\end{align*}
The second term in the inequality appears since 
\[
\d(w-v,x_1,y)\geq -\min\1\e_1r_0^2,\e_2\2
\]
 for any $x_1$ in $B_{r_0/2}(x_0)$ and $|y|\geq B_{r_0/2}$. Therefore,
\[
 L(w-v)(x_1) \geq -C\1\e_1 + \min(\e_1r_0^2,\e_2)\2.
\]
Then we choose $\e_2 = r_1^2/2$ and $\e_1$ sufficiently small to make $L(w-v)(x) \geq -\d/2$ uniformly in $L \in \cL$ and $x \in B_{r_1}(x_0)$. This finally implies that $IU \geq 0$ and concludes the prove of the lemma.
\end{proof}

\begin{lemma}\label{barrier}
Let $\varphi(x)=\min(1, C(|x|-1)_+^\a)$, where $C$ and $\a$ has been chosen as in \cite{C2}. Then for any pair $\s$, $\t$ satisfying H1 we have
\[
\cM^+_{\tilde{\cL}_0}\varphi(x)\leq 0, \ \  x\in \R^n\sm B_1.
\]
Moreover,
\[
\cM^+_{\tilde{\cL}_0}\varphi(x)\leq -\d < 0, \ \  x\in B_2\sm B_1.
\]
\end{lemma}

\begin{proof}
Let $r_0$ and $\a$ the radius and exponent from Lemma 3.1 in \cite{C2}. We know that $\cM_\s^+ v(x_0) = -d$ ($d>0$) and $|D_\t|v(x_0) = e <\8$ for every $x_0 \in \p B_{1+r_0}$.

Let $s \in (0,1)$ and rescale $v$ by
\[
v_s(x) = s^{-\a}v(sx) = (|x|-s^{-1})^\a_+.
\]
Recall the scaling remarks on section \ref{VSP}. For $L \in \tilde\cL_0(\s,\t,\l,\L,b)$ with kernel $K$ we consider $L_s \in \tilde\cL_0(\s,\t,s^{-\s}\l,s^{-\s}\L,s^{-\t}b)$ with kernel $K_s(y) = s^nK(sy)$ such that
\[
L_sv_s(x) = s^{-\a}(Lv)(sx).
\]

Let $x$ such that $(1+r_0)x = (1 + sr_0)x_0$ and translate $v_s$ such that it remains below $v$ but touches it in a whole ray passing through $x$ and $x_0$. We still denote the translation $v_s$. By the scaling,
\[
Lv(x) = s^\a(L_sv_s)(x_0) \leq (L_sv_s)(x_0)
\]
By the monotonicity lemma \ref{mono} applied to $v$ and $v_s$ at $x_0$,
\[
L_sv_s(x_0) \leq L_sv(x_0).
\]
Since $L_s \in \tilde\cL_0(\s,\t,s^{-\s}\l,s^{-\s}\L,s^{-\t}b) \ss \cL_0(\s,\t,s^{-\s}\l,s^{-\s}\L,s^{-\t}b)$,
\begin{align*}
L_sv(x_0) &\leq s^{-\s}\cM^+_\s v(x_0) + s^{-\t} b|D_t|v(x_0),\\
&\leq s^{-\s}\3-d+s^{\s-\t}be\4,\\
&\leq -d/2,
\end{align*}
if $s$ is small enough. By transitivity $Lv(x)\leq -d/2$ for any $x$ with $|x| = 1+sr_0$ with $s\in(0,s_0)$. Finally we just multiply $v$ by a constant $C$ big enough such that $C v(s_0r_0) \geq 1$ and use Lemma \ref{minlema} to conclude the lemma for the truncation of $v$.
\end{proof}

\begin{proof}[Proof of Theorem \ref{exist}]
Let $S$ be the set of all viscosity subsolutions of $Iv \geq 0$ with boundary data smaller than $g$,
\begin{align*}
S = \3v \in USC(\W)\cap L^\8(\R^n): Iv \geq 0 \text{ in viscosity in } \W \right.\\
\left.\text{ and $v \leq g$ in $\R^n \sm \W$}\4.
\end{align*}
The set $S$ is non empty because the constant function $u = -\|g\|_\8$ satisfies $Iu = 0$ given that $I$ is of the inf-sup type.

The first Lemma assures us that $\bar u$, defined as the upper semicontinuous envelope in $\W$ of $u(x) = \sup_{v\in S}v(x)$, is a viscosity sub solution of $I\bar u \geq 0$. Then $\bar u \in S$ and $\bar u = u$ is a sub solution too. By the second lemma the lower semicontinuous envelope $\underbar{u}$, is a super solution. If not that would contradict the fact that $u$ is the biggest subsolution. We conclude, by the comparison principle, that $\underbar{u}\geq u$ and therefore both have to be equal and $u$ is a viscosity solution of $Iu = 0$ in $\W$.

The next step is to prove that we actually attain the boundary values in a continuous way. We have to show that for any $x \in \R^n \sm \W$ and any $\e > 0$ we can find continuous barriers $v$ and $w$ such that,
\begin{enumerate}
\item[(i)] $Iw \leq 0$ and $Iv \geq 0$ in $\W$ in the viscosity sense,
\item[(ii)] $w \geq g$ and $v \leq g$ in $\R^n \sm \W$,
\item[(iii)] $w(x) \leq g(x) + \e$ and $v(x) \geq g(x) - \e$.
\end{enumerate}
We just prove it for $w$.

If $x$ belongs to the interior of $\R^n \sm \W$ then a function $w$ which is equal to $\|g\|_\8$ for every $y \neq x$ and equal to $g(x)$ for $y = x$ is in $USC(\W)$ and is a super solution. If $x \in \p\W$ then there is a ball $B_{r_0}(x+r_0\eta)$ such that $\bar{B}_{r_0}(x+r_0\eta)\cap\partial \W=\3 x\4$, where $\eta$ is a unitary vector and $r_0$ less than one. Let
\[
w(y) = 2\|g\|_\8\varphi\1\frac{y-(x+r\eta)}{r}\2 + g(x) + \e
\]
with $\varphi$ from Lemma \ref{barrier} and some $r<r_0$. By the construction of $\varphi$ we already have that (i) and (iii) are satisfied.

To check (ii) let $\d>0$ such that $|g(y)-g(x)| \leq \e$ whenever $|x-y|\leq \d$. Take $r$ such that $B_{2r}(x+r\eta) \ss B_{\d}(x)$. If $y \in B_{\d}(x) \cap (\R^n \sm \W)$ then $w(y) = g(x) + \e \geq g(y)$. If $y \in \R^n \sm B_{\d}(x) \ss \R^n \sm B_{2r}(x+r\eta)$ then $\varphi \geq 1$ and $w(y) \geq \|g\|_\8 \geq g(y)$.
\end{proof}

%%%%%%%%%%%%%%%%%%%%%%%%%%
%%%%%% Partial ABP %%%%%%%
%%%%%%%%%%%%%%%%%%%%%%%%%%

\section{Partial ABP Estimates}\label{ABPp}

The classical ABP theorem states that for a super solution, positive in $\p B_3$, the supremum of $u^-$ is controlled by the $L^n$ norm of the right hand side, integrated only over the contact set for the convex envelope. These estimates are useful to get lower bounds in the measure of the contact set which are then needed to get point estimates.

We denote by $\G$ the convex envelope supported in $B_3$. For a lower semicontinuous function $u\geq 0$ in $\R^n\sm B_1$,
\[
\G(x) = \sup\{v(x): v:B_3\to\R \text{ is convex and } v \leq u^-\}.
\]
We get the same definition if $v$ is only affine. Every time we refer to $\n\G(x)$ we are actually referring to a sub differential of $\G$ at $x$ which always exists. 

In the next lemma we see that we can almost put a paraboloid above $\G$, with the opening controlled by $f(x)$, the supremum of $u$ outside $B_1$ and the $\t$ derivative of $\G$ at $x$.

\begin{lemma}\label{preabp}
Let $u \geq 0$ in $\R^n \sm B_1$ be a globally bounded viscosity solution of,
\begin{align*}
\cM^-_{\tilde{\cL}_0}u \leq f \text{ in } B_1,
\end{align*}
and $x \in \3u = \G\4$. Assume H1 holds ($2>\s>\s_0$ and $\min(1,\s)>\t>\t_0$) and $2b \leq \l(2-\s)/(1-\t)$. Let $\r_0 = 1/(128\sqrt{n})$, $r_k = \r_02^{-1/(2-\s)-k}$ and $R_k = B_{r_k} \sm B_{r_{k+1}}$. Then there is a constant $C_0$ such that for any $M > 0$ there is a $k$ such that
\[
|\3y\in R_k: u(y+x) > u(x) + y\cdot \n\G(x) + Mr_k^2\4| \leq C_0 \frac{F(x)}{M}|R_k|,
\]
where
\[
F(x) = f(x) + (1-\t)b \int_{B_2} \frac{|\d_0(\G,x;y)|}{|y|^{n+\t}} dy + \frac{1-\t}{\t}b \|u^+\|_{L^{\8}(\R^n \sm B_1)}.
\]
\end{lemma}

\begin{proof}
Notice that $\d_e^-(u,x;y) = 0$. If $x\pm y \in B_3$ then we use that there is a plane touching $u$ by below in $B_3$. If $x+y \notin B_3$ then $x-y\notin B_1$ and the boundary value gives that $u(x\pm y) \geq 0$ and then $\d_e(u,x;y)\geq 0$ because $u(x) \leq 0$.

By Lemma \ref{welldefined} the following quantities can be computed and satisfy,
\begin{align*}
&\int_{B_{r_0}} (2-\s)\l \frac{\d_e^+(u,x;y)}{|y|^{n+\s}} - (1-\t)b \frac{|\d_o(u,x;y)|}{|y|^{n+\t}} dy \leq\\
&f(x) + (1-\t)b \int_{\R^n \sm B_{r_0}} \frac{|\d_o(u,x;y)|}{|y|^{n+\t}} dy \leq\\
&f(x) + C(n)b\frac{1-\t}{\t}\|u^+\|_{L^{\8}(\R^n \sm B_1)}.
\end{align*}

We want to use Lemma \ref{interpolation} with $\Gamma$ as the test function. The assumption $2b \leq \l(2-\s)/(1-\t)$ guarantees \eqref{interpol} with $\a=1/2$,
\begin{align*}
&\int_{B_{r_0}} (2-\s)\l \frac{\d_e^+(u,x;y)}{|y|^{n+\s}} - (1-\t)b \frac{|\d_o(u,x;y)|}{|y|^{n+\t}} dy \geq\\
&\int_{B_{r_0}} \frac{(2-\s)\l}{2}\frac{\d_e^+(u,x;y)}{|y|^{n+\s}} dy - (1-\t)b \int_{B_{r_0}} \frac{|\d_0(\G,x;y)|}{|y|^{n+\t}} dy.
\end{align*}
Adding what we have so far
\begin{align}\label{abp1}
(2-\s)\int_{B_{r_0}} \frac{\d_e^+(u,x;y)}{|y|^{n+\s}} dy \leq CF(x)
\end{align}

The rest of the proof goes as in \cite{C1}. Fix $M$ and assume that none of the dyadic rings satisfies the conclusion of the lemma for $C_0$ still to be fixed. For every $y \in R_k$ where
\[
u(y+x) > u(x) + y\cdot\n\G(x) + Mr_k^2,
\]
we have that,
\begin{align*}
\d_e(u,x;y) &= u(x+y) + u(x-y) - 2u(x),\\
&> y\cdot\n\G(x) + Mr_k^2 + u(x-y) - u(x),\\
&\geq Mr_k^2.
\end{align*}
because by the convexity of $\G$,
\[
- y\cdot\n\G(x) + u(x) = - y\cdot\n\G(x) + \G(x) \leq \G(x-y) \leq u(x-y).
\]
Adding all the contributions into the estimate \eqref{abp1} to get,
\begin{align*}
(2-\s)\int_{B_{r_0}} \frac{\d_e^+(u,x;y)}{|y|^{n+\s}} &\geq (2-\s)C_0M\frac{F(x)}{M}\sum_{k=0}^\8r_k^{2-\s},\\
&\geq C_0F(x)r_0^{2-\s}\frac{2-\s}{1-2^{-(2-\s)}}.
\end{align*}
Now it just a matter to take $C_0$ large enough to get a contradiction. Notice that the quotient $(2-\s)/(1-2^{-(2-\s)})$ is uniformly bounded by above and away from zero when $\s$ varies in $(0,2)$.
\end{proof}

The following is just a modification of the previous lemma. The aim is to replace the second term in $F(x)$ by $\|u\|_{\8}$.

\begin{remark}\label{star}
By the intermediate value theorem, for each $x \in B_1$ and $y \in B_2$, $|\d_0(\G,x;y)|$ is equal to $2|\n\G(x')\|y|$ for $x'$ an intermediate point in the segment between $x+y$ and $x-y$. So that
\[
\int_{B_2} \frac{|\d_0(\G,x,y)|}{|y|^{n+\t}}dy \leq \frac{C(n)}{1-\t}\|\n\G\|_{\8}.
\]
By the geometry of the convex envelope $\|\n\G\|_{\8} \leq \|u^-\|_\8/2$. We can also consider that $\t > \t_0 > 0$ for $\t_0$ universal, so that $F(x)$ can be simplified to
\[
F(x) = f(x) + b\|u\|_{\8}.
\]
Notice that we haven't absorb the constant $b$ into the universal constants of the estimate. The importance of this choice will be seen in the results of the next sections.
\end{remark}

%%% corollary pre abp %%%%

\begin{corollary}\label{coropreabp}
 Under the assumptions of Lemma \ref{preabp} there exists a small fraction $\e_0>0$ and a constant $M_0>0$ such that for some radius $r$ and $R = B_r \sm B_{r/2}$:
\begin{enumerate}
 \item[(i)] $|\3y\in R: u(y+x) > u(x) + y\cdot \n\G(x) + M_0 F(x)r^2\4| \leq \e_0|R|$.
 \item[(ii)] $\G(y+x) \leq u(x) + y\cdot \n\G(x) + M_0 F(x)r^2$ for $y \in B_{r/2}$.
 \item[(iii)] $|\n\G(B_{r/4}(x))| \leq C|B_{r/4}|F(x)^n$
\end{enumerate}
\end{corollary}

\begin{proof}
 Let $A$ be the following set
\[
 A = \1B_1 \sm B_{1/2}\2 \cap \{x_1 > 1/2\}.
\]
Take
\[
\e_0 = \frac{|A|}{2|B_1 \sm B_{1/2}|} \text{ and } M_0 = \frac{C_0}{\e_0},
\]
with $C_0$ from Lemma \ref{preabp}. Apply Lemma \ref{preabp} with $M = F(x)M_0$ to get a radius $r (= r_k)$ such that
\begin{align}\label{cororing}
 |\3y\in R: u(y+x) > u(x) + y\cdot \n\G(x) + M_0 F(x)r^2\4| \leq \e_0|R|.
\end{align}

By convexity we can assume without loss of generality that $\G$ attains its maximum $N$ on $B_{r/2}(x)$ at the point $(r/2)e_1 + x$ and
\[
 \G(y+x) \geq \G((r/2)e_1 + x),
\]
for every $y \in R$ with $y\cdot e_1 \geq r/2$. Therefore,
\begin{align*}
2\e_0|R| &\leq |\{y\in R:\G(y+x) \geq N\}|,\\
&\leq |\{y\in R: u(y+x)\geq N\}|.
\end{align*}
Then $N$ has to be smaller or equal than $u(x) + y\cdot \n\G(x) + M_0 F(x)r^2$ because otherwise we get a contradiction with \eqref{cororing}. This implies (ii).

Finally, by $\G$ being trapped between two planes in $B_{r/2}$, separated by a distance $M_0F(x)r^2$, we get by the geometry of convex functions a control in the oscillation of $\n\G$ in $B_{r/4}$. Namely $\n\G(B_{r/4})$ is contained in the ball of radius $4M_0 F(x)r$ with center at $\n\G(x)$. This concludes the proof.
\end{proof}

Now we are able to state and prove an ABP type estimate.

%%% abp %%%

\begin{theorem}\label{ABP}
Let $u \geq 0$ in $\R^n \sm B_1$ be a globally bounded viscosity solution of,
\begin{align*}
\cM^-_{\tilde{\cL}_0}u \leq f \text{ in } B_1.
\end{align*}
Assume H1 holds and $2b \leq \l(2-\s)/(1-\t)$. There is a disjoint family of cubes $Q_j$ with diameters $d_j \leq \r_02^{-1/(2-\s)}$ ($\r_0 = 1/(32\sqrt{n})$) which covers the contact set $\3\G = u\4$ such that the following holds
\begin{enumerate}
 \item[(i)] $\3u = \G\4 \cap \bar{Q}_j \neq \emptyset$ for any $Q_j$.
 \item[(ii)]\label{abp2}$\displaystyle\left|\3y \in 8\sqrt{n}Q_j: u(y) < \G(y) + C \1\max_{x \in Q_j \cap \3\G = u\4}F(x)\2d_j^2\4\right| \geq \mu|Q_j|$.
 \item[(iii)]$\displaystyle |\n\G(\bar{Q}_j)| \leq C \1\max_{x \in Q_j \cap \3\G = u\4}F(x)\2^n|Q_j|$.
\end{enumerate}
where $\mu$ ($=(1-\e_0)$ from Corollary \ref{coropreabp}) and $C$ above are universal (independent of $\s$ and $\t$) and
\[
F(x) = f(x) + (1-\t)b \int_{B_2} \frac{|\d_0(\G,x;y)|}{|y|^{n+\t}} dy + \frac{1-\t}{\t}b \|u^+\|_{L^{\8}(\R^n \sm B_1)}.
\]
\end{theorem}

\begin{proof}
Lets proceed as in \cite{C1} and cover $B_1$ with a tiling of cubes of diameter $\r_02^{-1/(2-\s)}$. We discard all those that do not intersect the contact set $\{u=\G\}$. Whenever a cube does not satisfy (ii) and (iii), we split it into $2^n$ congruent cubes of half diameter and discard those whose closure does not intersect $\{u=\G\}$. We want to prove that eventually this procedure finishes.

Let's assume that the covering process does not stop. We end up getting a sequence of nested cubes intersecting at a point $x_0 \in \{u=\G\}$. We will prove that there is a cube in the family that did not split, reaching then a contradiction.

Due to Corollary \ref{coropreabp} there is a radius $0<r<\r_02^{-1/(2-\s)}$ such that for $R=B_r\sm B_{r/2}$,
\[
 |\3y\in R: u(y+x) > u(x) + y\cdot \n\G(x) + CF(x_0)r^2\4|\leq \e_0|R|,
\]
and
\[
 |\n\G(B_{r/4}(x_0))|\leq CF(x_0)^n|B_{r/4}|.
\]

There is a cube $Q_j$ with diameter $r/8 \leq d_j < r/4$ such that, $B_{r/4}(x_0)\supset \bar{Q}_j$ and $B_r(x_0)\subset 32\sqrt{n} Q_j $. 

Using the fact that the diameter of the cube and the radius are comparable and that, by the convexity of $\G$, $\G(y)\geq u(x_0)+(y-x_0)\cdot\n\G(x_0)$, we get
\begin{align*}
&\left|\3y \in 32\sqrt{n}Q_j: u(y) \leq \G(y) + C\1\max_{Q_j \cap \3\G = u\4}F\2d_j^2\4\right| \geq\\
&\left|\3y \in 32\sqrt{n}Q_j: u(y) \leq u(x_0)+(y-x_0)\cdot\n\G(x_0) + C\1\max_{Q_j \cap \3\G = u\4}F\2d_j^2\4\right|\\
&\geq (1-\e_0)|R|\geq \m|Q_j|.
\end{align*}
This is (ii) in the statement of the Theorem. Since $\bar{Q}_j$ is contained in $B_r$ we conclude also that (iii) holds and $Q_j$ did not split.
\end{proof}

As $\t$ and $\s$ go to one and two respectively in a controlled way (recall the hypothesis $2b\leq \l(2-\s)/(1-\t)$), this theorem recovers a sufficient step to complete the proof of the classical ABP estimate. However, to prove regularity for $u$ it will be sufficient to use a weaker version where $F(x) = f(x) + b\|u\|_{\8}$ (see Remark \ref{star}).

We also point out that the condition on $b$ is not too restrictive. By the scaling discussion on Section \ref{VSP} we can always consider a dilation of $u$ to make the assumption valid.

%%%%%%%%%%%%%%%%%%%%%%%%%%%%
%%%%%% Point Estimate %%%%%%
%%%%%%%%%%%%%%%%%%%%%%%%%%%%

\section{Point Estimate}\label{PE}

The point estimate for non linear operators works in someway like the mean value theorem for super harmonic functions. If a non negative super harmonic function is bigger or equal than 1 in half of the points in $B_1$ (in measure) then it gets automatically separated from zero at $B_{1/4}$ a fixed quantity. This is the key step to prove a decay of oscillation and then H\"older regularity for the solutions of our equations. For non local operators, point estimates were already given in \cite{S1}. Those estimates are easier to obtain than in the local case because the definition of the non local operators already involve some sort of averaging. However, the estimates in \cite{S1} blow up when the order of the equation go to the classical one. Our goal here is to see that the same estimates still hold with constant that remain uniform when $\s\to2$ and $\t\to1$ in a controlled way.

From this point on we will always assume that, for $\s_0,\t_0,m,A_0>0$ given, the set of hypothesis H1, H2 and H3 holds.

\begin{enumerate}
\item[(H1)] $2>\s\geq\s_0>0$, $\min(1,\s)>\t\geq\t_0>0$,
\item[(H2)] $\s-\t \geq m > 0$,
\item[(H3)] $\l A_0(2-\s) \geq b(1-\t)$.
\end{enumerate}

We recall the special function constructed in \cite{C1}.

\begin{lemma}
Let $2>\s_0>0$, there is a function $\Phi$ such that,
\begin{enumerate}
 \item[(i)]   $\Phi$ is continuous in $\R^n$,
 \item[(ii)]  $\Phi(x) = 0$ for $x$ outside $B_{2\sqrt n}$,
 \item[(iii)] $\Phi(x) < -2$ for $x$ in $Q_3$, and
 \item[(iv)]  $\cM^+_\s\Phi \leq \psi(x)$ in $\R^n$ for some non negative function $\psi(x)$ supported in $\bar B_{1/4}$
\end{enumerate}
for every $\s > \s_0$.
\end{lemma}

The following lemma provides the first and also the inductive step towards an inductive proof of the point estimate.

\begin{lemma}
Let $\s_0,\t_0,m,A_0>0$ and assume H1, H2 and H3. There exists constants $\m\in(0,1)$, $\e_0>0$ and $M>1$, such that if 
\begin{itemize}\label{iteracion1}
\item[(i)] $u\geq 0$ in $\R^n$,
\item[(ii)] $\inf_{Q_{3\k}}u\leq 1$,
\item[(iii)] $\cM^-_{\tilde{\cL}_0} u\leq 1$ in $Q_{4\sqrt{n}\k}$,
\end{itemize}
then
\[
|\3u\leq M\4 \cap Q_{\k}| > \m|Q_{\k}|,
\]
for
\[
\k = \frac{\e_0}{(1 + \|u\|_{\8})^{1/(\s-\t)}}.
\]
\end{lemma}

\begin{proof}
Consider $\tilde u(x)=u(\k x)$ and note that by the scaling of the equation $\tilde u$ satisfies (i), (ii) in the cube of side 3 and
\[
\cM_{\tilde{\cL}_0(\tilde b)}u\tilde u \leq \k^\s \text{ in } Q_{4\sqrt n},
\]
where $\tilde{\cL}_0(\tilde b)=\tilde{\cL}_0(\s,\t,\l\,\L,\tilde b)$, for $\tilde{b} = \k^{\s-\t}b$. We will prove that the lemma holds for $\tilde u$ in $Q_1$, which implies the desired result. The proof follows as in \cite{C1} but we point out that the ABP type results that we have are different.

First thing we require from $\e_0^m$ is to be small enough, with respect to $A_0$, such that the condition of smallness on $\tilde b\leq \e_0^mb$ from Lemma \ref{preabp} and Theorem \ref{ABP} holds. Namely, $2\e_0^m \leq A_0^{-1}$.

Consider $v=\tilde u + \Phi$, where $\Phi$ is the special function given in \cite{C1}. We have that $v$ satisfies in $Q_{4\sqrt n}$ (since $\tilde{\cL}_0(\tilde{b})\ss\cL_0(\tilde{b})$)
\begin{align*}
\cM_\s^-v - \tilde{b}|D_\t| v &\leq \k^\s + \cM_\s^+\Phi + \k^{\s-\t}b|D_\t|\Phi,\\
                               &\leq \k^\s + \psi + \k^{\s-\t}bC,
\end{align*}
for a universal constant $C$.

Let $\G$ be the concave envelope of $v$ supported in the ball $B_{6\sqrt n}$. Let $Q_j$ be the cubes from a rescaled version of our partial ABP estimate.
We have
\begin{align*}
\max_{B_{2\sqrt n}} v^- &\leq C|\nabla \Gamma (B_{2\sqrt n})|^{1/n}\leq C\left(\sum_j|\nabla \Gamma(\bar Q_j)|\right)^{1/n},\\
                        &\leq C\1\sum_j\1\max_{Q_j}\psi + \k^\s + \k^{\s-\t}b\1 1 + \|v\|_\8\2\2^n\2|Q_j|^{1/n}.
\end{align*}

We can make the terms $\k^\s + \k^{\s-\t}b\1 1 + \|v\|_\8\2$ small enough by choosing $\e_0$ small enough,
\[
\k^\s + \k^{\s-\t}b\1 1 + \|v\|_\8\2 \leq C(\e_0^\s + \e_0^{\s-\t}) \leq C\e_0^m.
\]

Using that $\Phi \leq -2$ in $Q_3$
\[
1 \leq C\e_0^m+C\left(\sum_{j}\1\max_{Q_j}\psi^+\2^n|Q_j|\right)^{1/n},
\]
which implies then, for $\e_0$ small enough, the following inequality,
\[
\frac{1}{2}\leq C\1\sum_{j}\1\max_{Q_j}\psi^+\2^n|Q_j|\2^{1/n}.
\]

Since $\psi$ is supported in $\bar{B}_{1/4}$ and is bounded, we get
\begin{align}\label{summeas}
\sum_{Q_j\cap\bar{B}_{1/4}\neq\emptyset}|Q_j| \geq c,
\end{align}
where $c$ is universal. Now, the diameters of all cubes $Q_j$ are bounded by $\rho_02^{-1/(2-\s)}$, which is smaller than $\rho_0=1/(128\sqrt n)$. So, every time we have that $Q_j$ intersects $B_{1/4}$ the cube $32\sqrt n Q_j$ will be contained in $B_{1/2}$.

Since $\e_0$ is universal, the partial ABP estimates translate into
\begin{align}\label{meas}
|\3x\in 32\sqrt nQ_j:\ v(x)\leq \Gamma(x)+Cd^2_j\4|\geq c|Q_j|,
\end{align}
for $C$ universal and $Cd_j^2<C\rho_0^2$. Let us consider now the cubes $32\sqrt n Q_j$ for every $Q_j$ that intersects $B_{1/4}$. This provides an open cover of the union  of the corresponding cubes $\bar{Q}_j$ and it is contained in $B_{1/2}$. Taking a subcover with finite overlapping and using \eqref{summeas} and \eqref{meas} we get
\[
|\3x\in B_{1/2}: \ v(x)\leq\Gamma(x)+C\rho_0^2\4|\geq c.
\]
Hence, if we let $-M_0=\min_{B_{1/2}}\Phi$ we get
\[
|\3x\in B_{1/2}: \ \tilde u(x)\leq M_0+C\rho_0^2\4| \geq c.
\]
Finally let $M=M_0+C\rho_0^2$ and $\m=c$. Since $B_{1/2}\subset Q_1$,
\[
|\3x\in Q_1: \ \tilde u(x)\leq M\4|\geq c,
\]
which concludes the result for $\tilde u$.
\end{proof}

\begin{remark}\label{alterhyp}
In the previous proof the scaling is necessary to have:
\begin{enumerate}
\item[(i)] $\tilde b (1 + \|u\|_\8) \leq \e_0^m$,
\item[(ii)] Right hand side smaller than $\e_0^m$.
\end{enumerate}
In future references we will use that if these identities hold, then the conclusion also holds without any further scaling.
\end{remark}

\begin{remark}\label{dilation}
Consider for $u$ and $\tilde{u}$ as before and $0 < r \leq 1$, $v(x)=\tilde u(rx)$. Then $v$ satisfies,
\[
\cM_{\tilde{\cL}_0(b(r\k)^{\s-\t})}v\leq (r\k)^\s,
\]
and in particular
\[
\cM_\s^-v(x) - b(r\k)^{\s-\t}|D_\t v| \leq (r\k)^\s.
\]
From the previous remark we check that $b(r\k)^{\s-\t}(1+\|v\|_\8) \leq \e_0^m$ and that the right hand side is also smaller or equal than $\e_0^m$ if $r \leq 1$.
\end{remark}

In particular, the transformations required to prove the full $L^\e$ lemma are of the form
\[
v(y) = \frac{u(x_0 + 2^{-i}y)}{M^k}.
\]
The previous lemma can still be applied and we can iterate by means of a Calder\'on Zygmund decomposition as in \cite{CC}.

\begin{lemma}\label{iteracionk}
Let $\s_0,\t_0,m,A_0>0$ and $u \geq 0$, $\tilde u$ as in Lemma \ref{iteracion1}. Then we have
\[
\left|\3\tilde{u}>M^k\4\cap Q_1\right|\leq(1-\mu)^k,
\]
for $k=1,2,...$ where $M$ and $\mu$ are as in Lemma \ref{iteracion1}. As a consequence we have the following inequality,
\[
|\3\tilde u > t \4\cap Q_1|\leq dt^{-\e}, \ \ \forall t>0,
\]
where $d$ and $\e$ are positive universal constants.  
\end{lemma}

By standard covering arguments one can pass from cubes to balls.

\begin{corollary}\label{le1}
Let $\s_0,\t_0,m,A_0>0$ and $u \geq 0$, $\tilde u$ as in Lemma \ref{iteracion1} with $u$ super solution of $\cM^-_{\tilde{\cL}_0}u \leq 1$ in $B_{2\k}$ and $u(0) \leq 1$. Then we have
\[
|\3\tilde u > t\4\cap B_1|\leq ct^{-\e}, \ \ \forall t>0,
\]
where $c$ and $\e$ are positive universal constants.
\end{corollary}

By using Remark \ref{dilation} one more time we can prove a rescaled version of the Corollary \ref{le1}.

\begin{corollary}\label{leR}
Let $\s_0,\t_0,m,A_0>0$, $u \geq 0$, $\tilde u$ as as in Lemma \ref{iteracion1} with $u$ super solution of $\cM^-_{\tilde{\cL}_0} u\leq C_0$ in $B_{2\k r}$, for $r \leq 1$. Then we have
\begin{align}\label{point}
|\3 \tilde u > t \4 \cap B_r | \leq C r^n (u(0)+C_0r^\s)^{\e} t^{-\e},  \ \ \forall t>0,
\end{align}
where $C$ and $\e$ are positive universal constants.
\end{corollary}

%%%%%%%%%%%%%%%%%%%%%%%%%%%%%%%
%%%%%% Holder Regularity %%%%%%
%%%%%%%%%%%%%%%%%%%%%%%%%%%%%%%

\section{H\"older Regularity}\label{Ho}

The first lemma in this section is the decay of oscillation, that follows from the point estimate already proved, applied at every scale. It is well known that when the oscillation of a function decays geometrically in geometrically decaying balls it implies a H\"older modulus of continuity at the center of such ball. By applying it at every point of a ball strictly contained in the domain we get H\"older regularity.

We still assume the general hypothesis, H1, H2 and H3, of the previous section.

\begin{lemma}
Let $\s_0,\t_0,m,A_0>0$ and assume H1, H2 and H3. Let $u$ be a function such that:
\begin{enumerate}
\item[(i)] $ \frac{-1}{2}\leq u\leq \frac{1}{2}$ in $\R^n$,
\item[(ii)] $\cM^+_{\tilde{\cL}_0}u \geq -1$ and $\cM^-_{\tilde{\cL}_0}u \leq 1$ in $B_\k$,
\end{enumerate}
in the viscosity sense. Let $\tilde u(x) = u(\k x)$ for
\[
\k = \frac{\e_1}{(1 + \|u\|_{\8})^{1/(\s-\t)}}.
\]
Then there are universals $\alpha,C>0$ such that
\[
|\tilde u(x)-\tilde u(0)|\leq C|x|^\alpha.
\]
\end{lemma}

Our proof relies in noticing that a dilation powerful enough puts us in the same hypothesis as in the proof of \cite{C1}. The detail is that the rescaling considered in such proof consists of a dilation of the domain, which as we already saw are good for our situation, times some constants that grow geometrically which compete against the smallness condition on the coefficient $b$. We want to check that by making $\a$ small enough we can control the effect of this second multiplication.

\begin{proof}
Let $\e_1^m \leq \e_0^m/2$ to start such that the estimates from the previous section are valid with the same constants.
 
We will show that there exists sequences $m_k$ and $M_k$ such that $m_k \leq \tilde u \leq M_k$ in $B_{4^{-k}}$ and 
\[
M_k-m_k=4^{-\a k}
\]
so that result holds for $C=4^\a$.

For $k=0$ we choose $m_0=-1/2$ and $M_0=1/2$ and by (i) we have $m_0 \leq \tilde u \leq M_0$ in $\R^n$. We construct the sequence by induction. Assume then that we have the sequences up to some $k$ and then we want to find $m_{k+1}$ and $M_{k+1}$. In the ball $B_{4^{-(k+1)}}$, either $\tilde u \geq (M_k+m_k)/2$ in at least half the points (in measure) or we have the other inequality. Let's assume, without loss of generality, that
\[
\left|\left\{\tilde u \geq\frac{M_k+m_k}{2}\right\}\cap B_{4^{-(k+1)}}\right|\geq\frac{|B_{4^{-(k+1)}}|}{2}.
\]

Consider now 
\[
v(x)=\frac{\tilde u(4^{-k}x)-m_k}{(M_k-m_k)/2},
\]
so that $v \geq 0$ in $B_1$ and $|\{v \geq 1\}\cap B_{1/4}|\geq|B_{1/4}|/2$.

From the inductive hypothesis, we have that for any index $j$ between 1 and $k$,
\begin{align*}
v\geq\frac{m_{k-j}-m_k}{(M_k-m_k)/2} &\geq \frac{m_{k-j}-M_{k-j}+M_k-m_k}{(M_k-m_k)/2}\\
                                     &\geq 2(1-4^{\a j}),
\end{align*}
in $B_{4^j}$. Therefore $v(x)\geq -2(|4x|^\a -1)$ outside $B_1$. Let $w =v^+$, it satisfies
\begin{align}\label{eqnca}
&\cM^-_{\tilde{\cL}_0(4^{-k(\s-\t)}\k^{\s-\t}b)}w\leq \\
\nonumber &4^{-k\a}\k^\s + \cM_\s^+v^- + 4^{-k(\s-\t)}\k^{\s-\t}b|D_\t|v^-\leq\\
\nonumber &\e_1^m + \cM_\s^+v^- + \e_1^m b|D_\t|v^-.
\end{align}
We still have $|\{w\geq 1\}\cap B_{1/4}|\geq |B_{1/4}|/2$. Use the other bound $v^- \leq 2(|4x|^\a - 1)$ outside $B_1$, also proved by induction, and $v^- = 0$ in $B_1$ to get that the right hand side can be made smaller than $\e_0^m$ in $B_{3/4}$ by choosing a small exponent $\a$.

We recall the conditions in the Remark \ref{alterhyp}. So far we have shown the second one which is satisfied with a right hand side $\e_0^m$. For the first condition note that
\begin{align*}
4^{-k(\s-\t)}\k^{\s-\t}b(1+\|w\|_\8) &\leq 4^{-k(\s-\t)}\e_0^m(1+4^{\a k}),\\
&\leq 4^{-k(m-\a)}\e_0^m.
\end{align*}
so we have to choose $\a \leq m$.

Now, given any $x\in B_{1/4}$ we can apply Corollary \ref{leR} in $B_{1/2}(x)$ to get
\[
C(w(x)+\e_1)^\e\geq|\{w > 1\}\cap B_{1/2}(x)|\geq\frac{|B_{1/4}|}{2},
\]
hence, since $\e_1$ can be made even smaller, we conclude $w \geq \theta>0$ in $B_{1/4}$ for some $\theta>0$. If we let $M_{k+1}=M_k$ and $m_{k+1}=m_k+\theta(M_k-m_k)/2$ we have the inductive step
\[
m_{k+1}\leq \tilde u \leq M_{k+1}, \ \ \hbox{in} \ B_{4^{-(k+1)}}.
\]
Moreover $M_{k+1}-m_{k+1}=(1-\theta/2)4^{-\alpha k}$, so choosing $\alpha$ and $\theta$ such that $1-\theta/2=4^{-\a}$ we conclude $M_{k+1}-m_{k+1}=4^{-\alpha(k+1)}$.
\end{proof}

As a result we get that $\tilde u$ is $C^\a$. Here is the proof of Theorem \ref{ca}.

\begin{proof}
Let $u$ be as in Theorem \ref{ca} and consider now
\[
v(x)=\frac{u(x)}{2(\|u\|_\8+C_0)}.
\]
Since the equation is homogeneous of degree $1$ we are now under the hypothesis of the previous lemma. Hence we conclude that the dilation $\tilde v$ of $v$ satisfies
\[
|\tilde{v}(x)-\tilde{v}(0)|\leq C|x|^\a,
\]
where $C$ is a universal constant. Coming back to $v$, this translates to
\begin{align*}
|v(x)-v(0)|&\leq C|x|^\a\k^{-\a}\\
&=\frac{C}{\e_0^\a}(\|v\|_\8+1)^{\a/(\s-\t)}|x|^\a,	
\end{align*}
since $\|v\|_\8\leq 1/2$, $\s-\t\geq m$ and $\e_0$ is universal we get
\[
|v(x)-v(0)|\leq C|x|^\a, 
\]
for a different universal $C$. In terms of $u$ we recover the estimate
\[
|u(x)-u(0)|\leq C(\|u\|_\8+C_0)|x|^\a.
\]
Hence $u$ is $C^\a$ at $0$ and it's $C^\a$ seminorm is controlled as desired. This concludes the proof.
\end{proof}

%%%%%%%%%%%%%%%%%%%%%%%%%%%%%%%%%
%%%%%% C^1,alpha estimates %%%%%%
%%%%%%%%%%%%%%%%%%%%%%%%%%%%%%%%%

\section{$C^{1,\a}$ Regularity}\label{CoA}

For translation invariant equations, $C^{1,\a}$ regularity comes by proving $C^\a$ regularity for the incremental quotients of a given solution. This procedure allows to improve the regularity from $C^\a$ to $C^{2\a}$ and so forth all the way up to $C^{0,1}$ and then to $C^{1,\a}$, see \cite{CC}. We need to use the comparison principle to see that these incremental quotients satisfy a uniformly elliptic equation with bounded measurable coefficients and zero right hand side, for which we already have $C^\a$ estimates. The difficulty in this case is that we need, in each step, these incremental quotients to be uniformly bounded in $\R^n$. The previous regularity only guaranties this on $B_{r-\d}$, given that the equation is satisfied in $B_r$.

Recall the class $\cL_1 = \cL_1(\s,\t,\l,\L,b,\r_0) \ss\tilde{\cL}_0(\s,\t,\l,\L,b)$ of all possible linear operators $L$ with non negative kernels $K$ such that they satisfy \eqref{kernel1} and \eqref{kernel2}, and the following integrability assumption for some radius $\r_0$,
\[
\int_{\R^n \sm B_{\r_0}} \frac{|K(y)-K(y-h)|}{|h|}dy \leq C \ \ \text{every time} \ \ |h| < \frac{\rho_0}{2}.
\]

\begin{theorem}
Let $\s_0,\t_0,m,A_0>0$ and assume that H1, H2 and H3 holds. There is $\r_0>0$ small enough so that if $I$ is an elliptic operator of the inf-sup type as in \eqref{eqinfsup} with all the linear operators in $\cL_1$ and $u$ a bounded viscosity solution of $Iu = 0$ in $B_1$, then there is a universal $\a>0$ such that $u \in C^{1,\a}(B_{1/2})$ and
\[
\|u\|_{C^{1,\a}(B_{1/4})} \leq C\|u\|_\8
\]
for some universal $C > 0$.
\end{theorem}

\begin{proof}

Let $\bar\a$ the H\"older exponent obtained by Theorem \ref{ca} and assume that it is not the reciprocal of an integer by making it smaller if necessary. Let $\d = 1/(4[1/\bar\a])$. We want to see that, for $k=0,1,\ldots,[1/\bar\a]-1$, the estimate
\begin{align}\label{hypc1a}
 \|u\|_{C^{0,k\bar\a}(B_{3/4-k\d})} \leq C(k)\|u\|_\8,
\end{align}
implies the next estimate,
\begin{align}\label{claimc1a}
 \|u\|_{C^{0,(k+1)\bar\a}(B_{3/4-(k+1)\d})} \leq C(k+1)\|u\|_\8.
\end{align}

Fix a unit vector $e\in\R^n$ and $\eta$ a smooth cut-off function supported in $B_{(3/4 - k\d) - \d/4}$ and equal to one in $B_{(3/4 - k\d) -\d/2}$. For given $h \in (-\d/8,\d/8)$ we define the following incremental quotients
\begin{align*}
 w_h(x)   &= \frac{u(x+he) - u(x)}{|h|^{\bar\a k}},\\
 w_1^h(x) &= \frac{(\eta u)(x+he) - (\eta u)(x)}{|h|^{\bar\a k}},\\
 w_2^h(x) &= \frac{((1-\eta) u)(x+he) - ((1-\eta) u)(x)}{|h|^{\bar\a k}}.
\end{align*}

When $x \in B_{(3/4 - k\d)-\d/8}$, $|w_1^h(x)|$ is bounded above by $C(k,\eta)\|u\|_\8$. By interpolation and \eqref{hypc1a},
\begin{align*}
|w_1^h(x)| &\leq \|\eta u\|_{C^{0,k\bar\a}(B_{3/4-k\d})},\\
&\leq \|u\|_{C^{0,k\bar\a}(B_{3/4-k\d})} + \|u\|_\8\|\eta\|_{C^{0,k\bar\a}(B_{3/4-k\d})},\\
&\leq C(k,\eta)\|u\|_\8.
\end{align*}
If $x \in \R^n \sm B_{(3/4-k\d)-\d/8}$ then $w_1^h(x)$ just cancels.

By using that the equation is translation invariant we have that $u$ and $u(\cdot+he)$ satisfy equations in the same ellipticity family, \emph{with positive kernels}. Then $w_h$ also satisfy an equation in the same ellipticity family by Lemma \ref{difofsol}. The function $w_1^h$ satisfy a similar equation as $w_h$, the difference is on the right hand side introduced by the cut-off,
\begin{align*}\label{oldeq}
\cM^+_{\cL_1}w_1^h \geq - \cM^+_{\cL_1}w_2^h \ \ \text{and} \ \ \cM^-_{\cL_1}w_1^h \leq - \cM^-_{\cL_1}w_2^h.
\end{align*}
For $x \in B_{(3/4-k\d)- 3\d/4}$ the terms $|\cM^\pm_{\cL_1} w_2^h|$ are controlled by $\|u\|_\8$ by using that
\[
\int_{\R^n \sm B_{\r_0}} \frac{|K(y)-K(y-h)|}{|h|}dy \leq C \text{ every time $|h| < \frac{\rho_0}{2}$}.
\]
with $\r_0 = \d/8$. Indeed, for $L \in \cL_1$ with kernel $K$ and $x \in B_{(3/4-k\d)- 3\d/4}$ and $|y| \leq \d/8$, $w_2^h(x+y) = 0$ and
\begin{align*}
|Lw_2^h(x)| &= \left|\int w_2^h(x+y)K(y)dy\right|,\\
&= \left|\int_{\R^n\sm B_{\d/8}} \frac{(1-\eta)u(x+y+h)-(1-\eta)u(x+y)}{|h|^{\bar\a k}}K(y)dy\right|,\\
&= \left|\int_{\R^n\sm B_{\d/8}} (1-\eta)u(x+y)|h|^{1-\bar\a k}\frac{K(y)-K(y-h)}{|h|}dy\right|,\\
&\leq C\|u\|_\8.
\end{align*}

We get then the equations for $w_h^1$ in $B_{(3/4-k\d)-3\d/4}$
\begin{align*}
\cM^+_{\cL_1}w_1^h \geq C\|u\|_\8 \ \ \text{and} \ \ \cM^-_{\cL_1}w_1^h \leq -C\|u\|_\8.
\end{align*}
By applying Theorem \ref{ca} to $w_1^h$ from $B_{(3/4-k\d)-3\d/4}$ to $B_{3/4-(k+1)\d}$ we conclude that for a constant $C(k+1)$ independent of $h$,
\[
 \|w_h^1\|_{C^{0,\bar\a}(B_{3/4-(k+1)\d})} \leq C(k+1)\|u\|_\8.
\]
This implies the estimate \eqref{claimc1a} by using Lemma 5.6 in \cite{CC}.

From $k=[1/\bar\a]-1$ to $k+1 = [1/\bar\a]$ we get that $u$ is Lipschitz in $B_{3/4}$ with the estimate
\[
 \|u\|_{C^{0,1}(B_{3/4})} \leq C\|u\|_\8.
\]

By applying the previous step one more time to the Lipschitz quotient we conclude the theorem.
\end{proof}

{\bf Acknowledgment.} The authors would like to thank Luis Caffarelli for proposing the problem and for various useful discussions.
%%%%%%%%%%%%%%%%%%%%%%%%%%
%%%%%% BIBLIOGRAPHY %%%%%%
%%%%%%%%%%%%%%%%%%%%%%%%%%


\begin{thebibliography}{00}

\bibitem{BI}
G. Barles, C. Imbert.{\em Second-order elliptic integro differential equations: viscosity solutions theory revisited.} Annales de l'Institut Henri PoincarŽ. Analyse Non LinŽaire (2008), no. 3, 567-585.

\bibitem{BK}
R. F. Bass, M. Kassmann. {\em H\"older continuity of harmonic functions with respect to operators of variable order.} Communications in Partial Differential Equations (2005), no. 8, 1249-1259.

\bibitem{CC}
L. Caffarelli, X. Cabr\'e. {\em Fully nonlinear elliptic equations.} American Mathematical Society Colloquium Publications, 43. American Mathematical Society, Providence, RI, 1995. vi+104 pp. 

\bibitem{C1}
L. Caffarelli, L. Silvestre. {\em Regularity theory for fully nonlinear integro differential equations.} Communications on  Pure and Applied Mathematics (2009), no. 5, 597-638.

\bibitem{C2}
L. Caffarelli, L. Silvestre. {\em Regularity results for nonlocal equations by approximation} Archive for Rational Mechanics and Analysis (2011), no. 1, 59-88.

\bibitem{C3}
L. Caffarelli, L. Silvestre. {\em The Evans-Krylov theorem for non local fully non linear equations.} Annals of Mathematics (2011), no. 2, 1163-1187.

\bibitem{KAL1}
Y. C. Kim, K. A. Lee. {\em Regularity results for fully nonlinear integro differential operators with nonsymmetric positive kernels.} arXiv:1011.3565v2 [math.AP]

\bibitem{S1}
L. Silvestre. {\em H\"older estimates for solutions of integro-differential equations like the fractional Laplace.} Indiana Univ. Math. J. (2006), no. 3, 1155-1174.

\bibitem{Soner}
Soner, H. M. {\em Optimal control with state-space constraint II.} SIAM Journal on Control and Optimization (1986), no. 6, 1110-1122. 
\end{thebibliography}
\end{document}